\documentclass{amsart}

\usepackage{amsmath,amssymb,amsthm}

\usepackage{units} % This allows us to use nicefrac

\newtheorem{prop}{Proposition}[section]
\newtheorem{thm}[prop]{Theorem}
\newtheorem{lem}[prop]{Lemma}

\newtheorem{ass}[prop]{Assumption}

\theoremstyle{definition}

\newtheorem{defn}[prop]{Definition}

\newtheorem{rem}[prop]{Remark}

\newtheorem*{ack}{Acknowledgements}

\def\co{\colon\thinspace}

\newcommand{\oalpha}{\overline{\alpha}}
\newcommand{\halpha}{\hat{\alpha}}

\newcommand{\C}{\mathbb{C}}

\newcommand{\rmd}{\mathrm{d}}

\newcommand{\rme}{\mathrm{e}}
\newcommand{\bfe}{\mathbf{e}}

\newcommand{\tf}{\tilde{f}}

\newcommand{\tg}{\tilde{g}}

\newcommand{\rmi}{\mathrm{i}}

\newcommand{\tk}{\tilde{k}}

\newcommand{\muCZ}{\mu_{\mathrm{CZ}}}

\newcommand{\N}{\mathbb{N}}

\newcommand{\sfp}{\mathsf{p}}
\newcommand{\frakp}{\mathfrak{p}}
\newcommand{\tPhi}{\widetilde{\Phi}}

\newcommand{\Q}{\mathbb{Q}}

\newcommand{\R}{\mathbb{R}}
\newcommand{\frakR}{\mathfrak{R}}

\newcommand{\ttsl}{\mathtt{sl}}

\newcommand{\tu}{\tilde{u}}

\newcommand{\Veps}{V_{\varepsilon}}
\newcommand{\tv}{\tilde{v}}

\newcommand{\Z}{\mathbb{Z}}

\DeclareMathOperator{\Int}{Int}

\DeclareMathOperator{\Diffc}{Diff_{\mathrm{c}}}
\DeclareMathOperator{\Omegac}{\Omega_{\mathrm{c}}}

\begin{document}

\author[P.~Albers]{Peter Albers}
\address{Mathematisches Institut, Universit\"at Heidelberg,
Im Neuenheimer Feld 205, 69120 Heidelberg, Germany}
\email{palbers@mathi.uni-heidelberg.de}

\author[H.~Geiges]{Hansj\"org Geiges}
\address{Mathematisches Institut, Universit\"at zu K\"oln,
Weyertal 86--90, 50931 K\"oln, Germany}
\email{geiges@math.uni-koeln.de}

\author[K.~Zehmisch]{Kai Zehmisch}
\address{Mathematisches Institut, Universit\"at Gie{\ss}en, Arndtstra{\ss}e 2,
35392 Gie{\ss}en, Germany}
\email{kai.zehmisch@math.uni-giessen.de}

\title[Pseudorotations and Reeb flows]{Pseudorotations of the $2$-disc
and\\Reeb flows on the $3$-sphere}

\date{}

\begin{abstract}
We use Lerman's contact cut construction to find a sufficient condition
for Hamiltonian diffeomorphisms of compact surfaces to embed
into a closed $3$-manifold as
Poincar\'e return maps on a global surface of section for a Reeb flow.
In particular, we show that the irrational pseudorotations of the $2$-disc
constructed by Fayad--Katok embed into the Reeb flow of a dynamically
convex contact form on the $3$-sphere.
\end{abstract}



\maketitle

%%%%%%%%%%%%%%%%%%%%%%%%%%%%%%%%%%%%%%%%%%%%%%%%%%%%%%%%%%%%%%%%%%%%%%

\section{Introduction}
A \emph{global surface of section} for the flow of a smooth
non-singular vector field $X$ on a closed
$3$-dimensional manifold $M$ is an embedded compact surface
$\Sigma\subset M$ with the following properties:
\begin{itemize}
\item[(i)] Each component of the boundary $\partial\Sigma$
is a  periodic orbit of~$X$.
\item[(ii)] The interior $\Int(\Sigma)$ is transverse to~$X$,
and the orbit of $X$ through any point in $M\setminus\partial\Sigma$
intersects $\Int(\Sigma)$ in forward and backward time.
\end{itemize}
The Poincar\'e return map $\psi\co\Int(\Sigma)\rightarrow\Int(\Sigma)$
sends a point $p\in\Int(\Sigma)$ to the first intersection point in forward
time of the flow line of $X$ through~$p$. In general, $\psi$ need not extend
smoothly to a diffeomorphism of~$\Sigma$; if the
return time of the flow goes to infinity as one approaches $\partial\Sigma$,
the rescaled vector field with return time $2\pi$, say, will blow up
near~$\partial\Sigma$.

Global surfaces of section were introduced by Poincar\'e in the
context of celestial mechanics, allowing him to reduce the search for
periodic orbits in the $3$-body problem to finding periodic points
of the return map. The most celebrated instance of this
approach is Poincar\'e's last geometric theorem on
area-preserving twist maps of the annulus, as proved by
Birkhoff, see~\cite[Section~8.2]{mcsa17}. Hofer, Wysocki and
Zehnder~\cite{hwz98} developed holomorphic curves techniques for finding
global surfaces of section for Reeb flows, and they established the
existence of those surfaces for Hamiltonian flows on strictly convex energy
hypersurfaces in~$\R^4$. (In this context, there is an area form
on the surface of section preserved by the return map.)
This has provided fresh impetus for the study of the $3$-body problem;
see~\cite{affhk12,afkp12,schn18} for recent applications
of such global symplectic methods to this problem.

In this paper we study what in some sense is the
dual or converse problem. Our goal is to realise certain
Hamiltonian diffeomorphisms of compact surfaces with boundary
as the return map of a Reeb flow on a closed $3$-manifold.
Specifically, we are interested in achieving this
for the irrational pseudorotations constructed by Fayad--Katok~\cite{faka04}.

\begin{defn}
An \emph{irrational pseudorotation} is a diffeomorphism $\psi$ of $D^2$
with the following properties:
\begin{itemize}
\item[(i)] $\psi$ is area-preserving for the standard area form of~$D^2$.
\item[(ii)] $\psi$ has $0\in D^2$ as a fixed point, and no other
periodic points.
\end{itemize}
\end{defn}

Here is our first main result. For the definition of dynamical convexity,
see Section~\ref{subsection:muCZ}.

\begin{thm}
\label{thm:main}
Let $\psi\co D^2\rightarrow D^2$ be an irrational pseudorotation as
constructed by Fayad--Katok. Then there is a dynamically convex contact
form on the $3$-sphere $S^3$, inducing the
standard contact structure, whose Reeb flow has a disc-like surface of
section on which the return map equals~$\psi|_{\Int(D^2)}$.
\end{thm}

In particular, the Reeb flow has exactly two (simple) periodic orbits:
the boundary of the surface of section, and the one corresponding
to the fixed point $0$ of~$\psi$. By the work of Cristofaro-Gardiner and
Hutchings~\cite{cghu16}, two is the minimal number of periodic
Reeb orbits on any closed $3$-dimensional contact manifold. Also,
our construction produces a contact open book in the
sense of Giroux~\cite{giro02}, cf.~\cite[Section~4.4.2]{geig08}:
the binding is given by the boundary of the surface of section, and
the pages are the translates of this surface by the
Reeb flow, suitably reparametrised.

Given an open book on a $3$-manifold adapted to a contact structure
$\ker\alpha$, the Reeb flow preserves the area form on the
\emph{interior} of the pages
induced by $\rmd\alpha$. If the Reeb flow
is tangent to the binding (i.e.\ the common boundary of the
pages), this area form degenerates along the boundary.
So it is to be expected that we cannot work with
an embedding of a page smooth up to the boundary if we want to
realise a return map preserving the standard area form.
Indeed, our construction for proving Theorem~\ref{thm:main}
will produce a topological embedding $D^2\hookrightarrow S^3$
smooth only on the interior of the disc. This embedding differs
from a smooth embedding by a radial reparametrisation of the disc,
and the image is a smooth disc in~$S^3$.
The following definition is to be understood in the same vein.

\begin{defn}
When an area-preserving diffeomorphism $\psi\co \Sigma\rightarrow\Sigma$
can be realised, on $\Int(\Sigma)$, as the
Poincar\'e return map on a global surface
of section for a Reeb flow on a closed $3$-manifold~$M$, we say that $\psi$
\emph{embeds into a Reeb flow on~$M$}.
\end{defn}

\begin{rem}
Theorem~\ref{thm:main} is actually a corollary of the much more general
Theorem~\ref{thm:main-general} we are going to formulate
in Section~\ref{subsection:sufficient}.
We shall see there that any Hamiltonian diffeomorphism
$\psi\co\Sigma\rightarrow\Sigma$ embeds into a Reeb flow, subject to
a condition on the $\infty$-jet at the boundary $\partial\Sigma$ of
the Hamiltonian function generating~$\psi$.

For clarity of exposition, we proceed from the particular to the general.
That is, we first prove the embeddability of Hamiltonian diffeomorphisms
whose generating Hamiltonian is particularly well behaved
near~$\partial\Sigma$ (Proposition~\ref{prop:intro}). We then perform a
limit process to demonstrate Theorem~\ref{thm:main}. An inspection of
that proof will yield the general result alluded to above.

The condition on the $\infty$-jet of the Hamiltonian can be verified
directly, so it applies to Hamiltonian functions that do not necessarily
arise as a limit of `well-behaved' Hamiltonians, as is the case in
the Fayad--Katok examples.
\end{rem}

The pseudorotations of Fayad--Katok have precisely three
ergodic invariant measures: the Lebesgue measure on the disc, the
$\delta$-measure at the fixed point, and the Lebesgue measure
on the boundary. Thus, the Reeb flows we construct are in some sense
as exotic as possible. However, even disregarding the two periodic orbits,
the Reeb flow will not be minimal, since by the work of Le Calvez and
Yoccoz~\cite{lcyo97} there will always be other non-dense orbits.
We refer the reader to \cite[Section~3.1]{faka04} for further historical
comments. Concerning the minimality issue, see also the discussion
in~\cite{geze}.

\begin{rem}
Another construction of `exotic' Reeb flows is mentioned
in \cite[p.~200]{hwz98}. In private communication to those
authors, M.~Herman has constructed hypersurfaces in $\R^4$
that are $C^{\infty}$-close to an irrational ellipsoid and admit
precisely two periodic orbits, but have a Reeb flow with a dense orbit.

From the viewpoint of contact homology,
dynamically convex contact forms inducing the standard contact structure
on $S^3$, and whose Reeb flow has precisely two periodic orbits,
have been studied by Bourgeois--Cieliebak--Ekholm in \cite{bce07}.
They mention that in the context of their main theorem, there
is a disc-like global surface of section on which
the return map has a single fixed point and no further periodic
points, but they leave open the question whether pseudorotations
are actually realised in this way.

For a recent survey on global surfaces of section for Reeb flows
see~\cite{hrsa18}.
\end{rem}

Conversely, the embedding of the Fayad--Katok pseudorotations into a Reeb
flow on a closed manifold may pave the way to using global symplectic
methods for studying these pseudorotations. For recent applications of
pseudoholomorphic curves methods to the study of pseudorotations
see~\cite{bram15a,bram15b}.

The irrational pseudorotations of Fayad--Katok are $C^{\infty}$-limits
\[ \lim_{\nu\rightarrow\infty}\varphi_{\nu}\circ\frakR_{p_{\nu}/q_{\nu}}\circ
\varphi_{\nu}^{-1}\]
of conjugates of $2\pi$-rational rotations~$\frakR_{p_{\nu}/q_{\nu}}$,
where the conjugating maps
$\varphi_{\nu}$ are area-pre\-serv\-ing diffeomorphisms of $D^2$ that are
the identity on a small and, for $\nu\rightarrow\infty$,
shrinking neighbourhood of the
boundary $\partial D^2$. We shall describe these pseudorotations in
more detail later. In order to prove Theorem~\ref{thm:main},
we first establish the analogous statement for
area-preserving diffeomorphisms of $D^2$ that equal a rigid
rotation near the boundary. Such a result is essentially contained in
\cite{hutc16} or \cite[Section~3]{abhs18}. We present an
alternative proof that relies on the notion of contact cuts in the sense
of Lerman~\cite{lerm01}.

Contact cuts provide the natural language
for constructing contact \emph{forms} on manifolds obtained from
a manifold with boundary by collapsing the orbits of an $S^1$-action on the
boundary, allowing one to control the Reeb dynamics on such quotients.
Therefore the cut construction is ideally suited for formulating
the general condition on a Hamiltonian
diffeomorphism to embed into a Reeb flow. For a brief introduction
to contact cuts in the context of Reeb dynamics see~\cite{geig19}.

As an instructive first step towards the general result, with
this approach one easily sees how one can relax the condition that the
diffeomorphism be a rigid rotation near the boundary, as in the
following proposition.

\begin{prop}
\label{prop:intro}
Let $\psi$ be the time $2\pi$ map of a Hamiltonian isotopy of $D^2$
generated by a $2\pi$-periodic Hamiltonian function
$H_s\co D^2\rightarrow\R$, $s\in\R/2\pi\Z$. If the Hamiltonian function
is autonomous on a collar neighbourhood of $\partial D^2$ and
depends only on the radial coordinate in that neighbourhood,
then $\psi$ embeds into a Reeb flow on~$S^3$.
\end{prop}

This proposition will be given a short proof in Section~\ref{section:cuts},
after a discussion of contact cuts and their relation to contact open books.

In order to use Proposition~\ref{prop:intro} for proving
Theorem~\ref{thm:main}, in particular for the limit process in the
Fayad--Katok construction, we need to write the area-preserving
diffeomorphisms under consideration in a canonical fashion
as the time $2\pi$ map of a non-autonomous Hamiltonian function.
This is done in Section~\ref{section:area-disc}. The discussion there
includes a proof of the following result, which is probably folklore.

\begin{thm}
\label{thm:diff}
The space $\Diffc(D^2,\omega)$ of area-preserving diffeomorphisms
of $D^2$ with compact support in the interior $\Int(D^2)$
has $\{\mathrm{id}_{D^2}\}$ as a strong deformation retract.
\end{thm}

The proof of Theorem~\ref{thm:main} will be given in
Section~\ref{section:pseudorot}, except for the statement about
dynamical convexity, which will be established in
Section~\ref{section:invariants}, where we compute Conley--Zehnder indices
and other invariants of the Reeb flows we construct.

In Proposition~\ref{prop:conjugate} we shall see
that if $\psi$ embeds into a Reeb flow, then
so does its conjugate $\varphi^{-1}\circ\psi\circ\varphi$ under any
area-preserving diffeomorphism $\varphi$ of~$D^2$.
Strictly speaking, the embeddability property has to be formulated
for a pair $(H_s,\lambda)$, where $\lambda$ is a primitive of the
area form on~$D^2$. In Section~\ref{subsection:dF} we shall see that,
at least up to $C^2$-differentiability, the choice of
primitive is irrelevant.
\section{Contact open books as contact cuts}
\label{section:cuts}
In this section we are going to prove Proposition~\ref{prop:intro}.
We begin by describing the cut construction, and how it can be used
to construct open book decompositions. We then define a contact
form on the solid torus $S^1\times D^2$ whose Reeb flow gives
the solid torus the structure of a mapping torus of $(D^2,\psi)$,
where $\psi$ is the given Hamiltonian diffeomorphism. The desired
contact form on $S^3$ is then produced by a contact cut.
\subsection{Open books via the cut construction}
\label{subsection:open-book}
An \emph{open book decomposition} of a $3$-manifold $M$ consists of
a link $B\subset M$, called the \emph{binding}, and a smooth, locally trivial
fibration $\frakp\co M\setminus B\rightarrow S^1=\R/2\pi\Z$. It is
assumed that $\frakp$ is well behaved near the binding. By this we mean
that one can find a tubular neighbourhood $B\times D^2$ of $B$ in $M$
on which the map $\frakp$ is given by the angular coordinate in the
$D^2$-factor. The closures $\Sigma_s$ of the fibres $\frakp^{-1}(s)$,
$s\in S^1$, are called the \emph{pages}.
The binding is the common boundary of the pages.

Every closed, orientable $3$-manifold admits an open book
decomposition, see~\cite{rolf76}.

The vector field $\partial_s$ on $S^1$ lifts to a vector field on
$M\setminus B$ that coincides near $B$ with the angular vector field
on the $D^2$-factor of $B\times D^2$.
The time $2\pi$ flow of this vector field defines
a diffeomorphism $\psi$ of $\Sigma:=\Sigma_0$ to itself, equal to the
identity near the boundary $\partial\Sigma=B$. This diffeomorphism is called
the \emph{monodromy} of the open book.

Conversely, an open book can be built starting from
a compact surface $\Sigma$ with boundary, and a diffeomorphism $\psi$ of
$\Sigma$ that equals the identity near the boundary. This construction
is well known, see~\cite[Section 4.4.2]{geig08}. Here we are going
to interpret it as a cut construction in the sense of
Lerman~\cite{lerm01}, cf.\ \cite[Remark~5.6]{mnw13} and
\cite[Section~2.2.3]{doer14}.

This construction starts with the mapping torus
\[ V:=\Sigma\times [0,2\pi]/(x,2\pi)\sim(\psi(x),0)\]
of $(\Sigma,\psi)$. The boundary of $V$ is
$\partial V=\partial\Sigma\times S^1$. Write $\theta$, by slight abuse
of notation, for the
$S^1$-coordinate on the components of the boundary $\partial\Sigma$,
and $s$ for the $S^1$-coordinate on $V$ given by the projection onto
the second factor.

Consider the $S^1$-action on the boundary $\partial V$ of the mapping torus
generated by the vector field $\partial_s-h\partial_{\theta}$, where
$h$ is an integer. If $\partial V$ has several components
$\partial_iV$, $i=1,\ldots,k$, one may choose
an integer $h_i$ for each component. Let $M:=V/\!\!\sim$ be the
quotient space obtained by identifying points on $\partial V$ that lie
on the same $S^1$-orbit. The idea of Lerman's cut construction
is to identify this seemingly singular quotient space with the quotient
of a larger manifold under a free $S^1$-action. In the present setting,
the details will be given in the following proposition and its proof; for the
general construction see~\cite{lerm01}.

\begin{prop}
The space $M=V/\!\!\sim$ is a smooth closed $3$-manifold. It carries the
structure of an open book with binding
$B:=\bigl(\partial V/\!\!\sim\bigr)\cong\partial\Sigma$ and projection map
\[ \frakp\co M\setminus B=\Int(V)\longrightarrow S^1\]
given by the $s$-coordinate. The monodromy of the open book
equals the composition of $\psi$ with an $h_i$-fold right-handed Dehn
twist along a curve parallel to the boundary circle $\partial_iV$,
$i=1,\ldots, k$.
\end{prop}

\begin{proof}
Since we have to consider the components of $\partial V$ separately, we may
as well pretend that $\partial V$ is connected.
Write $(-\varepsilon,0]\times\partial\Sigma$ for a collar neighbourhood
of $\partial\Sigma$ in $\Sigma$ on which $\psi$ acts as the identity. Then
\[ \Veps:=(-\varepsilon,0]\times\partial\Sigma\times S^1\]
is a collar neighbourhood of $\partial V$ in~$V$. We think of $\Veps$
as a subset of the open bicollar
\[ N:=(-\varepsilon,\varepsilon)\times\partial\Sigma\times S^1.\]
Lift the $S^1$-action on $\partial V=\partial\Sigma\times S^1$ in the
obvious way to an $S^1$-action on~$N$. Then the function
$N\rightarrow (-\varepsilon,\varepsilon)$ assigning to each point
$(\tau,\theta,s)\in N$ its bicollar parameter
$\tau$ is smooth, $S^1$-invariant, and its $0$-level set $\partial V$
is regular. The function
\begin{equation}
\label{eqn:mu}
\begin{array}{rccc}
\mu\co & N\times\C         & \longrightarrow & \R\\
       & (\tau,\theta,s;z) & \longmapsto     & \tau+|z|^2
\end{array}
\end{equation}
is invariant under the anti-diagonal $S^1$-action
\begin{equation}
\label{eqn:action}
\rme^{\rmi\varphi}(\tau,\theta,s;z):=
(\tau,\theta-h\varphi,s+\varphi;\rme^{-\rmi\varphi}z),
\end{equation}
and $\mu^{-1}(0)$ is a regular level set on which the
$S^1$-action is free. It follows that $\mu^{-1}(0)/S^1$ is
a smooth manifold.

Observe that $\mu^{-1}(0)=P\times\partial V$, where $P$ is the
paraboloid
\[ P:=\bigl\{(\tau,z)\in(-\varepsilon,\varepsilon)\times\C\co
\tau=-|z|^2\bigr\}.\]
The $S^1$-action on the $P$-factor is free away from the apex $(0,0)$,
which is a fixed point of the action. It follows that taking the
quotient of $\mu^{-1}(0)$ under the $S^1$-action is the same
as forming the quotient space $\Veps/\!\!\sim$.
Thus, $M=V/\!\!\sim$ is a smooth manifold. The homeomorphism
\[ \bigl(\Veps/\!\!\sim\bigr)\longrightarrow\mu^{-1}(0)/S^1\]
induced by
\begin{equation}
\label{eqn:nbhd-B1}
\begin{array}{ccc}
\Veps           & \longrightarrow & \mu^{-1}(0)\\
(\tau,\theta,s) & \longmapsto     & (\tau,\theta,s;\sqrt{-\tau})
\end{array}
\end{equation}
defines the smooth manifold structure of $M$ near $B=\partial V/\!\!\sim$.

The manifold $\mu^{-1}(0)/S^1$ is diffeomorphic to $\partial\Sigma\times
\Int(D^2_{\sqrt{\varepsilon}})$, which can be seen as follows.
Consider the differentiable map
\[\begin{array}{ccc}
\mu^{-1}(0)       & \longrightarrow & \partial\Sigma\times
                                    \Int(D^2_{\sqrt{\varepsilon}})\\
(\tau,\theta,s;z) & \longmapsto     & (\theta+hs,\rme^{\rmi s}z).
\end{array}\]
Notice that on the left-hand side, $\tau$ is determined by $\tau=-|z|^2$.
Points on the same orbit of the $S^1$-action (\ref{eqn:action})
have the same image, so the map descends to
\[ \mu^{-1}(0)/S^1\longrightarrow\partial\Sigma\times
\Int(D^2_{\sqrt{\varepsilon}}). \]
This induced map is a diffeomorphism with inverse map
\begin{equation}
\label{eqn:nbhd-B2}
\begin{array}{ccc}
\partial\Sigma\times
\Int(D^2_{\sqrt{\varepsilon}}) & \longrightarrow & \mu^{-1}(0)/S^1\\
(b,\rho\rme^{\rmi\vartheta})   & \longmapsto     & 
      \bigl[(-\rho^2,b-h\vartheta,\vartheta;\rho)\bigr].
\end{array}
\end{equation}
This map is well defined even for $\rho=0$, since the points
$(0,b-h\vartheta,\vartheta;0)$ precisely make up the $S^1$-orbit through
the point $(0,b,0;0)$ as $\vartheta$ varies over~$S^1$.

This diffeomorphism identifies $B=\partial V/\!\!\sim$ with
$\partial\Sigma\times\{0\}$. The $S^1$-valued function
\[ (\tau,\theta,s;z)\longmapsto s+\arg z\]
on $\mu^{-1}(0)\setminus\{z=0\}$ is $S^1$-invariant, and under the
identification
\[ \Int(\Veps)\cong\bigl(\mu^{-1}(0)\setminus\{z=0\}\bigr)/S^1\]
coming from (\ref{eqn:nbhd-B1}), this function coincides with~$s$,
i.e.\ the fibration $\frakp$ defining the open book.
On the other hand, under the identification
\[ \partial\Sigma\times\bigl(\Int(D^2_{\sqrt{\varepsilon}})\setminus\{0\}\bigr)
\cong\bigl(\mu^{-1}(0)\setminus\{z=0\}\bigr)/S^1\]
coming from (\ref{eqn:nbhd-B2}), that function coincides with~$\vartheta$,
i.e.\ the angular coordinate in the disc factor.

It remains to determine the monodromy. On the mapping torus~$V$,
the monodromy $\psi$ is the return map on $\Sigma\times\{0\}$ given by
the flow $[(x,0)]\mapsto [(x,t)]$ at time~$2\pi$.
On the collar~$\Veps$, this flow is given by
\[ (\tau,\theta,s)\longmapsto(\tau,\theta, s+t),\]
and the return map is the identity. On the other hand,
on the neighbourhood $\partial\Sigma\times\Int(D^2_{\sqrt{\varepsilon}})$
of the binding, the monodromy should also be the
identity near $\rho=0$, realised as the time $2\pi$ map of the flow
\[ (b,\rho\rme^{\rmi\vartheta})\longmapsto(b,\rho\rme^{\rmi(\vartheta+t)})\]
in angular direction along the disc factor. Under the
identification of
\[ \partial\Sigma\times\bigl(\Int(D^2_{\sqrt{\varepsilon}})
\setminus\{0\}\bigr)\]
with $\Int(\Veps)$, this flow becomes
(near $\tau=0$)
\[ (\tau,\theta,s)\longmapsto(\tau,\theta-ht,s+t).\]
This implies that the monodromy on $\Veps$ has to be of the form
\[ (\tau,\theta,s)\longmapsto(\tau,\theta+\chi(\tau)t,s+t),\]
where $\chi$ interpolates smoothly between $0$ near $\tau=-\varepsilon$
and $-h$ near $\tau=0$. This amounts to an $h$-fold right-handed Dehn
twist along a $\theta$-circle, i.e.\ a boundary parallel curve.
\end{proof}
\subsection{Hamiltonian disc maps and contact forms}
\label{subsection:disc-contact}
The mapping torus of any orientation-preserving diffeomorphism $\psi$ of the
closed unit disc $D^2$ is a copy of the solid torus $S^1\times D^2$. Our aim
in this section is to construct contact forms on $S^1\times D^2$, starting
from a diffeomorphism $\psi$ that arises as the time $2\pi$ map of
a non-autonomous Hamiltonian. This construction is standard,
see~\cite{abhs18}. Much of our discussion generalises to
Hamiltonian diffeomorphisms of arbitrary compact, oriented surfaces with
boundary. We restrict attention to the $2$-disc, since this is
the case that will interest us later when we construct Reeb flows on $S^3$,
and it allows us to work with global coordinates.

Write $(r,\theta)$ for polar coordinates on the closed unit $2$-disc $D^2$.
As area form on $D^2$ we take $\omega:=2r\,\rmd r\wedge\rmd\theta$,
with primitive $1$-form $\lambda:=r^2\,\rmd\theta$.
Let $H_s$, $s\in S^1=\R/2\pi\Z$, be a $2\pi$-periodic Hamiltonian function
on~$D^2$. Throughout the present section, the following assumption, which
is part of the hypotheses in Proposition~\ref{prop:intro}, will be understood.

\begin{ass}
\label{ass:H}
There is a neighbourhood
of the boundary $\partial D^2$ in $D^2$ on which $H_s$ depends only on the
radial coordinate~$r$, not on $\theta$ or the `time' parameter~$s$.
\end{ass}

The Hamiltonian vector field $X_s$ is defined by
\[ \omega(X_s,\,.\,)=\rmd H_s.\]
This is the sign convention of~\cite{abhs18} and~\cite{mcsa17}, and it is
the one which is convenient in the present context.
By our assumption on $H_s$, the vector field $X_s$ will
be a multiple of the angular vector field $\partial_{\theta}$
near the boundary of $S^1\times D^2$. Without changing the Hamiltonian
vector field, we may assume that $H_s$ is as large as we like, and that
\begin{equation}
\label{eqn:boundary}
H_s|_{\partial D^2}=:h\in\N,
\end{equation}
by adding a positive constant to the Hamiltonian function.

\begin{lem}
For $H_s$ sufficiently large, the $1$-form
\[ \alpha:=H_s\,\rmd s +\lambda\]
is a positive contact form on $S^1\times D^2$.
Specifically, the condition for $\alpha$ to be a positive contact
form is given by
\begin{equation}
\label{eqn:contact}
H_s+\lambda(X_s)>0.
\end{equation}
\end{lem}

\begin{proof}
We compute
\begin{eqnarray*}
\alpha\wedge\rmd\alpha
  & = & \bigl(H_s\,\rmd s+\lambda\bigr)\wedge
        \bigl(\rmd H_s\wedge\rmd s+\omega\bigr)\\
  & = & \rmd s\wedge\bigl(H_s\omega+\lambda\wedge\rmd H_s\bigr).
\end{eqnarray*}
By adding a large constant to the Hamiltonian function, we can make
the first summand in parentheses large without changing the
second summand.

A word on notation is in order. When we write $\rmd H_s$, we mean
the differential of the function $H_s\co D^2\rightarrow\R$ for a fixed
value of the parameter~$s$, that is, there is no summand
$(\partial H_s/\partial s)\,\rmd s$.

With this understood, we have the identity
\[ \lambda\wedge\rmd H_s=\lambda(X_s)\cdot\omega,\]
which can be verified by taking the interior product with $X_s$
on both sides. (At points where $X_s=0$, the $2$-forms
on either side vanish.) It follows that the contact condition for $\alpha$
is equivalent to~(\ref{eqn:contact}).
\end{proof}

\begin{rem}
Since $\lambda$ equals the interior product of $\omega$ with
$r\partial_r/2$, we have $\lambda(X_s)=-\rmd H_s(r\partial_r/2)$,
so the contact condition (\ref{eqn:contact}) can equivalently be written as
\begin{equation*}
\label{eqn:contact-alt}
\tag{\ref{eqn:contact}'}
r\,\frac{\partial H_s}{\partial r}<2H_s.
\end{equation*}
\end{rem}

\begin{lem}
\label{lem:R}
When the contact condition (\ref{eqn:contact}) is satisfied, the
vector field
\[ R:=\partial_s+X_s\]
equals, up to positive scale, the Reeb vector field of~$\alpha$.
\end{lem}

\begin{proof}
We have
\[ i_R\rmd\alpha=i_R\bigl(\rmd H_s\wedge\rmd s+\omega\bigr)=
-\rmd H_s+\rmd H_s=0\]
and
\[ \alpha(R)=H_s+\lambda(X_s),\]
so the contact condition (\ref{eqn:contact}) is the same as $\alpha(R)>0$.
\end{proof}

\begin{lem}
\label{lem:invariant}
On a collar neighbourhood of $\partial (S^1\times D^2)$
in $S^1\times D^2$ where $H_s$ depends only on~$r$ and
$\partial H_s/\partial s=0$,
the contact form $\alpha$ is invariant under the
$S^1$-action generated by the vector field
$Y:=\partial_s-h\partial_{\theta}$.
\end{lem}

\begin{proof}
The Lie derivative of $\alpha$ with respect to $Y$ is, by the Cartan
formula,
\begin{eqnarray*}
L_Y\alpha
  & = & \rmd\bigl(\alpha(Y)\bigr)+i_Y\rmd\alpha\\
  & = & \rmd\bigl(H_s-hr^2\bigr)+i_Y\bigl(\rmd H_s\wedge\rmd s+\omega\bigr).
\end{eqnarray*}
Beware that in the first summand we also get a term
$(\partial H_s/\partial s)\,\rmd s$, but this term vanishes on
a collar neighbourhood of the boundary. In that neighbourhood,
where $H_s$ depends only on~$r$, we have $i_Y(\rmd H_s\wedge\rmd s)=
-\rmd H_s$. Then all terms in the expression for $L_Y\alpha$ cancel
in pairs.
\end{proof}

An $S^1$-action that preserves the contact form, not just the
contact structure, is called
a \emph{strict} contact $S^1$-action.
\subsection{Contact cuts}
\label{subsection:contactcuts}
Recall from \cite[Section~7.7]{geig08} that for a strict contact $S^1$-action
on a contact manifold $(N,\alpha)$ generated by a vector
field~$Y$, the \emph{momentum map} $\mu_N\co N\rightarrow\R$ is defined
as $\mu_N=\alpha(Y)$. From the identity
\begin{equation}
\label{eqn:reduce}
\rmd\mu_N=\rmd(\alpha(Y))=L_Y\alpha-i_Y\rmd\alpha=-i_Y\rmd\alpha
\end{equation}
it follows that the vector field $Y$ is tangent
to the level sets of~$\mu_N$. We also see that the level set $\mu_N^{-1}(0)$
is regular if and only if $Y$ is nowhere zero along this level. In that case,
the $S^1$-action is locally free on the $0$-level. If the action is
free, $\alpha$ induces a contact form on the quotient $\mu_N^{-1}(0)/S^1$.
This process is known as \emph{contact reduction}.
By (\ref{eqn:reduce}), the Reeb vector field of $\alpha$ is likewise
tangent to the level sets of~$\mu$, and it descends to the
Reeb vector field of the contact form on the reduced manifold.

The \emph{contact cut}, introduced by Lerman~\cite{lerm01}, produces a contact
form on the manifold obtained from the bounded manifold
$\mu_N^{-1}\bigl([0,\infty)\bigr)$ by collapsing the $S^1$-orbits on
the boundary $\mu_N^{-1}(0)$. Again, it is assumed that the $S^1$-action is
free on $\mu_{N}^{-1}(0)$.
This contact cut is constructed as follows. Consider the contact manifold
\[ \bigl( N\times\C,\alpha+x\,\rmd y-y\,\rmd x\bigr),\]
with circle action generated by $Y-(x\partial_y-y\partial_x)$.
The momentum map of this $S^1$-action is
\begin{equation}
\label{eqn:mu2}
\mu(p,z)=\mu_N(p)-|z|^2,\;\;\; (p,z)\in N\times\C.
\end{equation}
Then the reduced contact manifold $\mu^{-1}(0)/S^1$ is the desired cut.

Write $\pi\co\mu^{-1}(0)\rightarrow\mu^{-1}(0)/S^1$ for the projection onto
the orbit space. The contact form $\oalpha$ on the quotient is
characterised by
\[ \pi^*\oalpha=(\alpha+x\,\rmd y-y\,\rmd x)|_{T(\mu^{-1}(0))}.\]
It follows that the composition of maps
\begin{equation}
\label{eqn:strict-embed}
\begin{array}{ccccc}
\{p\in N\co\mu_N(p)>0\} & \longrightarrow & \mu^{-1}(0)
& \stackrel{\pi}{\longrightarrow} & \mu^{-1}(0)/S^1\\
p                       & \longmapsto     & \bigl(p,\sqrt{\mu_N(p)}\,\bigr)
& \longmapsto                     & \bigl[\bigl(p,\sqrt{\mu_N(p)}\,\bigr)\bigr]
\end{array}
\end{equation}
is an equidimensional strict contact embedding.

Likewise, the embedding
\[ \begin{array}{ccc}
\mu_N^{-1}(0) & \longrightarrow & \mu^{-1}(0)\\
p             & \longmapsto     & (p,0)
\end{array} \]
induces a codimension $2$ strict contact embedding of reduced manifolds,
\[ \mu_N^{-1}(0)/S^1\longrightarrow\mu^{-1}(0)/S^1.\]
\subsection{Disc maps and contact cuts}
We now combine the themes of the two preceding sections. Start with
the solid torus $V=S^1\times D^2$ with contact form
$\alpha=H_s\,\rmd s+\lambda$,
subject to the contact condition~(\ref{eqn:contact}). (If you prefer, you
may work with a slight thickening $N$ of the bounded manifold~$V$, but
this is not essential.) As before, we choose a Hamiltonian
function $H_s$ that satisfies Assumption~\ref{ass:H}
and condition~(\ref{eqn:boundary}).

Then the vector field
$Y:=\partial_s-h\partial_{\theta}$ generates a strict contact $S^1$-action
near the boundary $\partial V$. Along this boundary,
the momentum map $\mu_V=\alpha(Y)=H_s-hr^2$ takes the value zero.

\begin{lem}
Subject to the contact condition (\ref{eqn:contact-alt}), the boundary
$\partial V$ is a regular component of the $0$-level set of the momentum
map~$\mu_V$.
\end{lem}

\begin{proof}
The contact condition gives
\begin{equation}
\label{eqn:regular}
\frac{\partial H_s}{\partial r}\Big|_{\{s\}\times\partial D^2}<2h,
\end{equation}
which implies $\rmd\mu_V(\partial_r)<0$ along $\partial V$.
\end{proof}

\begin{rem}
\label{rem:sign}
The contact condition implies $\mu_V>0$ on the
interior of $V$ near $\partial V$. So the definition of
the function $\mu$ in (\ref{eqn:mu}) accords with the one in (\ref{eqn:mu2})
up to a global minus sign.
\end{rem}

\begin{lem}
\label{lem:quotient}
The manifold $(S^1\times D^2)/\!\!\sim$ obtained by collapsing the
orbits of $Y=\partial_s-h\partial_{\theta}$ along the boundary
$\partial (S^1\times D^2)$ is the $3$-sphere~$S^3$.
\end{lem}

\begin{proof}
The map
\[\begin{array}{ccc}
S^1\times D^2 & \longrightarrow & S^3\subset\C^2\\
(s;r,\theta)  & \longmapsto     & \bigl(\sqrt{1-r^2}\,\rme^{\rmi s},
                                  r\rme^{\rmi(\theta+hs)}\bigr)
\end{array}\]
is an explicit description of the quotient map.
\end{proof}

\begin{rem}
Observe that the quotient map is not differentiable in~$r=1$.
Thus, strictly speaking, we have shown only that the quotient is
homeomorphic to~$S^3$. Thanks to the existence and uniqueness
of differential structures on topological $3$-manifolds, this
is not something to worry about.

The quotient map in the proof is obtained by parametrising the closed
northern hemisphere $S^2_+$ of $S^2=S^3\cap{(\R\times\C)}$
as the graph of the map $z\mapsto \sqrt{1-|z|^2}$ on the
closed unit disc in the equatorial plane $\{0\}\times\C$, and then rotating
the graph under the $S^1$-action $\rme^{\rmi s}(z_1,z_2)=
(\rme^{\rmi s}z_1,\rme^{\rmi hs}z_2)$.

If instead we parametrise $S^2_+$ by the stereographic projection
of the equatorial unit disc from the south pole, we obtain the
smooth quotient map
\[ (s;r,\theta)\longmapsto\biggl(\frac{1-r^2}{1+r^2}\,\rme^{\rmi s},
\frac{2r}{1+r^2}\,\rme^{\rmi(\theta+hs)}\biggr).\]
There are other quotient maps one could consider, and in what follows we
shall choose one that is adapted to the contact form in question. The
options correspond to different choices of the collar parameter in the
cut construction. We shall elaborate on this issue in
Section~\ref{subsection:nbhd-binding}.
\end{rem}

Now consider the contact form $\alpha=H_s\,\rmd s+\lambda$
on $S^1\times D^2$, subject to the contact condition~(\ref{eqn:contact}).
The contact cut construction yields a contact form $\oalpha$
on $S^3=(S^1\times D^2)/\!\!\sim$.

\begin{lem}
The contact structure $\ker\oalpha$ on $S^3$ is diffeomorphic
to the standard tight contact structure.
\end{lem}

\begin{proof}
By the contact condition (\ref{eqn:regular}),
on a collar neighbourhood $\Veps$ of $\partial V$ in
$V=S^1\times D^2$, the function $H_s-hr^2$ is strictly monotonically
decreasing in $r$, and by (\ref{eqn:boundary}) it takes the
value zero on the boundary. In particular, the function is
positive on $\Veps\setminus\partial V$.
Consider the map
\[ \begin{array}{ccc}
\Veps         & \longrightarrow & \C^2\\
(s;r,\theta)  & \longmapsto     & \bigl(\sqrt{H_s-hr^2}\,\rme^{\rmi s},
                                    r\rme^{\rmi(\theta+hs)}\bigr).
\end{array} \]
This, too, is a model for the quotient map $V\rightarrow V/\!\!\sim$
near $\partial V$. Again, the map is not smooth, but its image
is a piece of a smooth star-shaped hypersurface in $\C^2$.
(We expand on this point in Remark~\ref{rem:hyper}.)

The pull-back of the standard Liouville $1$-form
$\lambda_{\R^4}=r_1^2\,\rmd\theta_1+r_2^2\,\rmd\theta_2$ under this map
equals $\alpha|_{\Veps}$. So the restriction
of $\lambda_{\R^4}$ to the hypersurface describes the
contact form on the contact cut $S^3=V/\!\!\sim$ near the
circle $\partial V/\!\!\sim$. Notice that the quotient map
identifies this circle with the unit circle in $\{0\}\times\C$,
no matter what choice of~$H_s$.

For the contact form $\bigl(1+(h-1)r^2\bigr)\,\rmd s+r^2\,\rmd\theta$,
the quotient map is the one in the proof of Lemma~\ref{lem:quotient},
with image $S^3\subset\R^4$.
The contact condition (\ref{eqn:contact-alt}) is convex
in~$H_s$. Thus, the convex linear interpolation between the given $H_s$
and $1+(h-1)r^2$ (and the corresponding interpolation of
starshaped hypersurfaces)
induces a smooth homotopy of contact forms on~$S^3$.
The result then follows from Gray stability~\cite[Theorem~2.2.2]{geig08}.
\end{proof}

\begin{rem}
\label{rem:hyper}
As promised, here is the argument why the image of $\Veps$ under the
quotient map is a smooth star-shaped hypersurface in $\C^2$. 
Smoothness is only an issue near $r=1$. There,
by the inverse function theorem, $r$ is a smooth function of $H_s-hr^2$.
This means that the points
\[ \bigl(\sqrt{H_s-hr^2}\,\rme^{\rmi s}, r\bigr) \]
on the hypersurface, corresponding to $\theta=hs$, form a smooth
surface of revolution, with $\sqrt{H_s-hr^2}$ playing the role of the
radius, and $r$ a function of that radius squared.

The $3$-dimensional hypersurface is then obtained by rotating this
surface under the $S^1$-action $\rme^{\rmi\varphi}(z_1,z_2)=
(z_1,\rme^{\rmi\varphi}z_2)$; in other words, we think of
$\theta$ as $\theta=hs+\varphi$. Since $|z_2|\in (1-\varepsilon,1]$
is bounded away from zero on the surface, this rotation produces a smooth
hypersurface.

To see that the hypersurface is star-shaped with respect to
the origin in $\C^2$, it is enough to observe that
the value of $|z_2|$ of image points increases with
increasing~$r$ (and $s,\theta$ fixed), while that of $|z_1|$
decreases by the contact condition~(\ref{eqn:regular}).
Alternatively, one can reach the same conclusion
by observing that $\lambda_{\R^4}$ pulls back to a contact form
on the hypersurface. Since the radial vector field
$(r_1\partial_{r_1}+r_2\partial_{r_2})/2$ is a Liouville
vector field for $\omega_{\R^4}=\rmd\lambda_{\R^4}$, the hypersurface
mast be transverse to the radial vector field.
\end{rem}
\subsection{Proof of Proposition~\ref{prop:intro}}
The Reeb vector field of $\alpha+x\,\rmd y-y\,\rmd x$, which is simply the
pull-back of the Reeb vector field $R_{\alpha}$ from $V=S^1\times D^2$
to $N\times\C$, descends to the Reeb vector field of the induced contact
form on the $3$-sphere $(S^1\times D^2)/\!\!\sim$. On $\Int(V)$, this
coincides with the old Reeb vector field~$R_{\alpha}$ by the strict contact
embedding~(\ref{eqn:strict-embed}), and hence up to positive
scale with $R=\partial_s+X_s$ by Lemma~\ref{lem:R}. So the inclusion
$\{0\}\times D^2\subset S^1\times D^2$ descends to the desired
embedding $D^2\hookrightarrow S^3$, smooth on $\Int(D^2)$.

On $\partial V$ we have $R=\partial_s+a\partial_{\theta}$
for some $a\in\R$. The contact condition (\ref{eqn:contact}), which we have
seen to be equivalent to $\alpha(R)>0$, translates along the boundary into
$h+a>0$. When we take the quotient of $\partial V$ with respect to
the $S^1$-action generated by $Y=\partial_s-h\partial_{\theta}$,
the vector field $R$ descends to $(h+a)\partial_{\theta}$
on $\partial V/S^1=\partial D^2$. The time $2\pi$ flow of this vector field
coincides with that of $R$, regarded as a map of
$\{0\}\times\partial D^2$ to itself.

This completes the proof of Proposition~\ref{prop:intro}.
\subsection{Contact structures supported by open books}
Let $M$ be a closed, oriented $3$-manifold with an open book
decomposition $\frakp\co M\setminus B\rightarrow S^1$. The standard
orientation of $S^1$ defines a coorientation of the fibres
$\frakp^{-1}(s)$; with the orientation of $M$ this determines
the positive orientation of the pages. The binding $B$
is endowed with the orientation as boundary of the pages.

A contact structure $\xi=\ker\alpha$ on
$M$ is said to be \emph{supported} by the open book if
the following compatibility conditions are satisfied:
\begin{itemize}
\item[(o)] The $3$-form $\alpha\wedge\rmd\alpha$ is a positive
volume form on~$M$.
\item[(i)] The $2$-form $\rmd\alpha$ induces a positive area form on
each fibre of~$\frakp$.
\item[(ii)] The $1$-form $\alpha$ is positive on each component of the
link~$B$.
\end{itemize}

As shown by Giroux~\cite{giro02}, every contact structure on a closed,
oriented $3$-manifold is supported by an open book.

The contact form on $S^3$ constructed in the proof of
Proposition~\ref{prop:intro} is supported by an open book with disc-like
pages. Condition (i) is guaranteed by the transversality of $R$
to the disc factor in $\Int (S^1\times D^2)$. The orientation
condition in (ii) is satisfied thanks to $h+a>0$.
\section{Area-preserving diffeomorphisms of the disc}
\label{section:area-disc}
In this section we want to describe how to make a sufficiently
canonical choice of Hamiltonian function generating any
given area-preserving diffeomorphism of~$D^2$
compactly supported in the interior. This, as mentioned earlier,
is essential for giving us the necessary control over the limit
process in the Fayad--Katok construction.

As before, we fix the area form $\omega=2r\,\rmd r\wedge\rmd\theta$
on $D^2$ with primitive $1$-form $\lambda=r^2\,\rmd\theta$.
Write $\Diffc(D^2,\omega)$ for the group of area-preserving
diffeomorphisms of $D^2$ with compact support in $\Int(D^2)$.
Similarly, $\Diffc(D^2)$ denotes the group of all diffeomorphisms
with the same condition on their support. The space of
area forms on $D^2$ of total area $2\pi$, and which
coincide with $\omega$ near $\partial D^2$, will be denoted
by~$\Omegac(D^2)$.

By Moser stability~\cite[Theorem~3.2.4]{mcsa17}, we have a Serre
fibration
\[ \begin{array}{ccccc}
\Diffc(D^2,\omega) & \longrightarrow & \Diffc(D^2)
    & \stackrel{\sfp}{\longrightarrow} & \Omegac(D^2)\\
                   &                 & f            
    & \longmapsto                      & f^*\omega,
\end{array}\]
cf.\ \cite[Lemma~1.1]{gima17}. The base $\Omegac(D^2)$ of this
fibration is contractible: convex linear interpolation between
any given area form and the base point $\omega$ defines a strong
deformation retraction to $\{\omega\}$. The total space
$\Diffc(D^2)$ is likewise contractible. As proved by
Munkres~\cite{munk60} and, independently, Smale~\cite{smal59}, it admits
a strong deformation retraction to $\{\mathrm{id}_{D^2}\}$.

From the Serre fibration property it then follows that the
fibre $\Diffc(D^2,\omega)$, too, is contractible.
In fact, as claimed in Theorem~\ref{thm:diff}, the fibre
also has $\{\mathrm{id}_{D^2}\}$ as a strong deformation retract.

\begin{proof}[Proof of Theorem~\ref{thm:diff}]
Write
\[ E_s\co \Diffc(D^2)\longrightarrow\Diffc(D^2),\;\;\; s\in [0,1],\]
for the strong deformation retraction of $\Diffc(D^2)$ to
$\{\mathrm{id}_{D^2}\}$, that is,
\begin{itemize}
\item[-] $E_0$ is the identity map on the space $\Diffc(D^2)$.
\item[-] $E_1$ maps the whole space to $\{\mathrm{id}_{D^2}\}$.
\item[-] $E_s(\mathrm{id}_{D^2})=\mathrm{id}_{D^2}$ for all $s\in [0,1]$.
\end{itemize}
Similarly, let
\[ B_t\co\Omegac(D^2)\longrightarrow\Omegac(D^2),\;\;\; t\in [0,1],\]
be the strong deformation retraction of $\Omegac(D^2)$ to~$\{\omega\}$.

Given $\psi\in\Diffc(D^2,\omega)$, the contraction $E_s$
defines a path $s\mapsto E_s(\psi)$ in the larger space $\Diffc(D^2)$.
This maps to a loop $\bigl(\sfp\circ E_s(\psi)\bigr)_{s\in [0,1]}$
in $\Omegac(D^2)$ based at~$\omega$.
The deformation retraction $B_t$ of $\Omegac(D^2)$
then defines a homotopy rel~$\{0,1\}$ from the constant loop at~$\omega$
to that loop $\bigl(\sfp\circ E_s(\psi)\bigr)_{s\in [0,1]}$, where
we take the retraction in backwards time:
\[ (s,t)\longmapsto \omega_{s,t}:=B_{1-t}\circ\sfp\circ E_s(\psi).\]
Notice that $\omega_{s,1}=\sfp\circ E_s(\psi)=E_s(\psi)^*\omega$.
Also, the $\omega_{s,t}$ coincide with $\omega$ in some neighbourhood
of the boundary~$\partial D^2$.

When one applies the Moser stability argument to the homotopy
$t\mapsto\omega_{s,t}$ (for each fixed~$s$), one needs to choose a family
of $1$-forms $\sigma_{s,t}$, compactly supported in $\Int(D^2)$, with
\[ \rmd\sigma_{s,t}=\frac{\rmd}{\rmd t}\omega_{s,t}.\]
Since
\[ \int_{D^2}\frac{\rmd}{\rmd t}\omega_{s,t}=\frac{\rmd}{\rmd t}
\int_{D^2}\omega_{s,t}=0,\]
such forms exist by the Poincar\'e lemma for compactly supported
cohomology. In Lemma~\ref{lem:poincare} below we make this explicit
in order to see that the $\sigma_{s,t}$ can be chosen canonically
and smoothly dependent on $s$ and~$t$.

Define the vector field $X_{s,t}$ on $D^2$ by
\[ \sigma_{s,t}+i_{X_{s,t}}\omega_{s,t}=0.\]
This is compactly supported in $\Int(D^2)$, so its flow $\psi_{s,t}$
(for each fixed~$s$) is defined for all times $t\in [0,1]$. By the usual
Moser argument, see~\cite[p.~108]{mcsa17}, this flow satisfies
$\psi_{s,t}^*\omega_{s,t}=\omega$.
Notice that $\omega_{0,t}=\omega=\omega_{1,t}$ for all $t\in [0,1]$.
This entails $\psi_{0,t}=\mathrm{id}_{D^2}=\psi_{1,t}$.

The map
\[ \begin{array}{rccc}
F_s\co & \Diffc(D^2,\omega) & \longrightarrow & \Diffc(D^2,\omega)\\
       & \psi               & \longmapsto     & E_s(\psi)\circ\psi_{s,1}
\end{array} \]
for $s\in [0,1]$ then defines
the desired strong deformation retraction of $\Diffc(D^2,\omega)$,
since $\psi_{s,1}^*E_s(\psi)^*\omega=\psi_{s,1}^*\omega_{s,1}=\omega$.
\end{proof}

It remains to discuss the canonical choice of the $1$-forms
$\sigma_{s,t}$. Here is the relevant version of the Poincar\'e
lemma for compactly supported forms. It shows that the $\sigma_{s,t}$
depend only on an \emph{a priori} choice of a bump function.
For simplicity of notation, we work on the unit square $I^2$,
with $I:=[0,1]$, instead of the unit disc.

\begin{lem}
\label{lem:poincare}
Choose a bump function $y\mapsto\chi (y)$ on $I$, compactly supported
in $\Int(I)$, with $\int_I\chi(y)\,\rmd y=1$.
Let $\eta=g(x,y)\,\rmd x\wedge\rmd y$ be a $2$-form on $I^2$
with $g$ compactly supported in $\Int(I^2)$ and $\int_{I^2}\eta=0$. Set
\begin{eqnarray*}
a(x) & := & \int_0^1 g(x,y)\,\rmd y,\\
b(x) & := & \int_0^x a(s)\,\rmd s,\\
u(x,y) & := & -g(x,y)+a(x)\chi(y),\\
v(x,y) & := & \int_0^y u(x,t)\,\rmd t.
\end{eqnarray*}
Then the $1$-form
\[ \beta:= v(x,y)\,\rmd x+b(x)\chi(y)\,\rmd y\]
is compactly supported in $\Int(I^2)$ and satisfies $\rmd\beta=\eta$.
\end{lem}

\begin{proof}
The fact that the functions $b$ and $v$ are compactly supported
in $I$ and $I^2$, respectively, follows from $\int_I a(x)\,\rmd x=0$
and $\int_I u(x,y)\,\rmd y=0$. The computation showing that $\beta$ is a
primitive of~$\eta$ is straightforward.
\end{proof}

We now want to show how the strong deformation retraction of
Theorem~\ref{thm:diff} translates into a canonical choice
of Hamiltonian function generating a given $\psi\in\Diffc(D^2,\omega)$.
Up to some sign changes and a little care concerning the
boundary behaviour, this is exactly the argument in
\cite[Proposition~9.3.1]{mcsa17}.
We shall assume that the strong deformation retraction
$F_s$ has been chosen as a \emph{technical} homotopy, i.e.\
$F_s$ is the identity map on $\Diffc(D^2,\omega)$
for $s$ near~$0$, and $F_s\equiv\mathrm{id}_D^2$ for $s$ near~$1$.

Given $\psi\in\Diffc(D^2,\omega)$, we define the path $s\mapsto\psi_s:=
F_{1-s}(\psi)$ in $\Diffc(D^2,\omega)$ from $\mathrm{id}_{D^2}$
to~$\psi$. Define the vector field $X_s$ on $D^2$ by
\[ \frac{\rmd}{\rmd s}\psi_s=X_s\circ\psi_s.\]
This vector field is compactly supported in $\Int(D^2)$, and $X_s\equiv 0$
for $s$ near $0$ or~$1$. There is a unique function $G_s\co D^2\rightarrow\R$
that is compactly supported in $\Int(D^2)$ and satisfies
\[ \psi_s^*\lambda-\lambda=\rmd G_s.\]
The function
\[ H_s:=-\lambda (X_s)+\Bigl(\frac{\rmd}{\rmd s}G_s\Bigr)\circ\psi_s^{-1}.\]
is compactly supported in $\Int(D^2)$, and it is identically
zero for $s$ near $0$ or~$1$, so it may be regarded as a $1$-periodic
function in~$s$. One then computes that $\rmd H_s=i_{X_s}\omega$,
so $\psi_s$ is the Hamiltonian isotopy generated by~$H_s$.
\section{Pseudorotations}
\label{section:pseudorot}
We now want to prove Theorem~\ref{thm:main} by performing a
limit process in the argument for proving Proposition~\ref{prop:intro}.
To this end, we need to describe pseudorotations as Hamiltonian
diffeomorphisms.
\subsection{Hamiltonian description of pseudorotations}
\label{subsection:pseudo-hamil}
Write $\frakR_a$ for the rotation of $D^2$ through an angle~$2\pi a$.
As mentioned in the introduction, the irrational pseudorotations
constructed by Fayad--Katok~\cite{faka04} are $C^{\infty}$-limits
\[ \lim_{\nu\rightarrow\infty}\varphi_{\nu}\circ\frakR_{p_{\nu}/q_{\nu}}
\circ\varphi_{\nu}^{-1},\]
where $(p_{\nu}/q_{\nu})_{\nu\in\N}$ is a sequence of rational numbers,
which we take to be positive,
converging sufficiently fast to a (Liouvillean) irrational
number, and the $\varphi_{\nu}$ are area-preserving
diffeomorphisms of $D^2$. Each $\varphi_{\nu}$ is the identity
on a neighbourhood of~$\partial D^2$. For $\nu\rightarrow\infty$, that
neighbourhood shrinks to~$\partial D^2$. The most relevant statements can
be found in Theorem~3.3 and Lemma~3.5 of~\cite{faka04}.

By the preceding section, where we now take our Hamiltonian isotopies to
be para\-metrised on the interval $[0,2\pi]$, we can write the
area-preserving diffeomorphism
\[ \varphi_{\nu}\circ\frakR_{p_{\nu}/q_{\nu}}\circ\varphi_{\nu}^{-1}
\circ\frakR_{p_{\nu}/q_{\nu}}^{-1}\in\Diffc(D^2,\omega)\]
in a canonical fashion as the time $2\pi$ map of a Hamiltonian isotopy
$\Psi_s^{\nu}$ generated by a $2\pi$-periodic Hamiltonian
function $K^{\nu}_s$ with compact support in $\Int(D^2)$. The rotation
$\frakR_{p_{\nu}/q_{\nu}}$ is the time $2\pi$ map of the Hamiltonian
isotopy generated by the function
\begin{equation}
\label{eqn:R-hamil}
R^{\nu}:=h+\frac{p_{\nu}}{q_{\nu}}-\frac{p_{\nu}}{q_{\nu}}r^2,
\end{equation}
where $h$ is chosen as a large natural number.
By the well-known formula for composing Hamiltonian diffeomorphisms,
see~\cite[Exercise~3.1.14]{mcsa17}, the diffeomorphism
\[ \psi_{\nu}:=\varphi_{\nu}\circ\frakR_{p_{\nu}/q_{\nu}}\circ
\varphi_{\nu}^{-1}=\Psi_{2\pi}^{\nu}\circ\frakR_{p_{\nu}/q_{\nu}}\]
is the time $2\pi$ map of the Hamiltonian isotopy generated by
\begin{equation}
\label{eqn:H}
H_s^{\nu}:= K_s^{\nu}+R^{\nu}\circ\bigl(\Psi_s^{\nu}\bigr)^{-1}.
\end{equation}
Observe that this Hamiltonian satisfies Assumption~\ref{ass:H}
and the boundary condition~(\ref{eqn:boundary}), since
$H_s^{\nu}=R^{\nu}$ near $r=1$.

By the discussion in the preceding section, the fact that the pseudorotation
$\psi:=\lim_{\nu\rightarrow\infty}\psi_{\nu}$ is a $C^{\infty}$-limit
implies that we have a $C^{\infty}$-limit
$H_s^{\infty}:=\lim_{\nu\rightarrow\infty}H_s^{\nu}$, and $\psi$
equals the time $2\pi$ map of the Hamiltonian isotopy generated by
this limit Hamiltonian~$H_s^{\infty}$.
The smoothness of $H_s^{\infty}$  follows from the fact that the
pointwise convergence $f_{\nu}\rightarrow f$ of a sequence
$(f_{\nu})_{\nu\in\N}$ of
$C^1$-functions and the uniform convergence of its partial
derivatives $\partial f_{\nu}/\partial x_i\rightarrow f_i$ implies,
by the fundamental theorem of calculus, that the limit function
$f$ is a $C^1$-function with $\partial f/\partial x_i=f_i$.
\subsection{The cut construction for circle actions on the boundary}
Our aim will be to show that the contact cut construction
in Section~\ref{section:cuts} can be performed for the contact
form $\alpha_{\infty}:=H_s^{\infty}\,\rmd s +\lambda$ on $V=S^1\times D^2$.
Notice that $H_s^{\infty}$ still satisfies the
boundary condition~(\ref{eqn:boundary}), but it violates
Assumption~\ref{ass:H}, in general.

This means that the $S^1$-action on $\partial (S^1\times D^2)$
defined by $Y=\partial_s-h\partial_{\theta}$
may not extend to a strict contact $S^1$-action on a collar
neighbourhood of the boundary. However, since the cut construction only
affects the boundary, one can sometimes perform a cut even
when the $S^1$-action does not extend. As we shall see, this is the
case here.

\begin{rem}
Topologically, one can always extend an $S^1$-action on the boundary
to one on a collar neighbourhood, and hence perform
a cut. In the symplectic setting,
one can appeal to an equivariant coisotropic embedding theorem
and conclude likewise, see~\cite[Proposition~2.7]{lerm01}.

In the contact setting, Giroux's neighbourhood theorem for
surfaces in contact $3$-manifolds, see~\cite[Theorem~2.5.22]{geig08},
or its higher-dimensional analogue~\cite[Proposition~6.4]{dige12},
gives an extension of the $S^1$-action to one
preserving only the contact \emph{structure}. By averaging the
contact form, one may assume it to be $S^1$-invariant, but this
would of course alter the Reeb dynamics.
\end{rem}

Since we are interested in preserving the Reeb dynamics on
$\Int(V)$, we shall explicitly analyse the $1$-form
on the quotient $V/\!\!\sim$ induced by the
contact form~$\alpha_{\infty}$ near the binding
$B:=\bigl(\partial(S^1\times D^2)\bigr)/S^1\cong\partial D^2$ and
discuss its extendability to the binding.
\subsection{The neighbourhood of the binding}
\label{subsection:nbhd-binding}
The diffeomorphism (\ref{eqn:nbhd-B2}) from
Sec\-tion~\ref{subsection:open-book},
for $\Sigma= D^2$, gives us an embedding $\Phi$ of a pointed
neighbourhood of the binding $B\cong\partial D^2$ into the interior of the
solid torus~$V=S^1\times D^2$. This embedding is given
by setting $\tau=-\rho^2$, so it depends on a choice of
collar parameter~$\tau=\tau(r,s,\theta)$. This function 
should be chosen to be invertible with respect to~$r$, that is, we
require that $r$ can be written as a smooth function $r=g(\tau,s,\theta)$.

Then the embedding $\Phi$ takes the form
\begin{equation}
\begin{array}{rccc}
\Phi\co & B\times
        \bigl(\Int(D^2_{\sqrt{\varepsilon}})\setminus\{0\}\bigr)
           & \longrightarrow & \Int(S^1\times D^2)\\[.5mm]
        & (b;\rho,\vartheta)
           & \longmapsto     & \left\{\begin{array}{l}
                               s=\vartheta;\\
                               r=g(-\rho^2,\vartheta,b-h\vartheta),\\
                               \theta=b-h\vartheta.
                               \end{array}\right.
\end{array}
\end{equation}

One obvious choice for the collar parameter is $\tau=r-1$.
Alternatively, one can choose
a collar parameter adapted to the contact form $\alpha=H_s\,\rmd s+\lambda$.
Here the natural collar parameter to use is the one coming from
the momentum map
\[ \mu_V=\alpha(Y)=H_s-hr^2.\]
The collar parameter $\tau$ would simply be the negative of that.

Thus, when we consider the sequence of contact forms
$\alpha_{\nu}:=H_s^{\nu}\,\rmd s+\lambda$ with limit
$\alpha_{\infty}$ on $S^1\times D^2$, we could opt to work
with a fixed collar parameter, or with one that changes with
each element in the sequence. We shall briefly describe the advantages
of either choice.
\subsubsection{Collar parameter depending on~$\alpha$}
\label{subsubsection:collar1}
We first consider the collar parameter
\[ \tau=\tau(r,s,\theta)=hr^2-H_s(r,\theta)\]
adapted to the contact form $\alpha=H_s\,\rmd s+\lambda$.
We have
\[ \frac{\partial\tau}{\partial r}\Big|_{r=1}=
2h-\frac{\partial H_s}{\partial r}\Big|_{r=1}>0\]
by the contact condition~(\ref{eqn:contact-alt}).
This means that near $r=1$, we can write $r$ as a smooth function
$r=g(\tau,s,\theta)$. Then
\[ \Phi^*\alpha=\bigl(g(-\rho^2,\vartheta,b-h\vartheta)\bigr)^2\,
(\rmd b-h\,\rmd\vartheta)+(H_s\circ\Phi)\,\rmd\vartheta.\]
We have
\[ (H_s-hr^2)\circ\Phi(b,\rho,\vartheta)=
-\tau\bigl(g(-\rho^2,\vartheta,b-h\vartheta),\vartheta,b-h\vartheta\bigr)
=\rho^2,\]
hence
\[ \Phi^*\alpha=\bigl(g(-\rho^2,\vartheta,b-h\vartheta)\bigr)^2\,
\rmd b+\rho^2\,\rmd\vartheta.\]
The second summand obviously extends smoothly over the binding
$\{\rho=0\}$, so the only question is whether the function
\[ (b;\rho,\vartheta)\longmapsto \bigl(g(-\rho^2,\vartheta,b-h\vartheta)
\bigr)^2\]
extends smoothly. When it does, the extended $1$-form is
easily seen to be a contact form.

Notice that in the case where the Hamiltonian function satisfies
Assumption~\ref{ass:H}, in which case the $S^1$-action on the boundary of
the solid torus extends to a collar neighbourhood,
$\tau$ is a function of $r$ only near $r=1$,
and hence $g$ is a function of $\rho$ only. So in this case the
contact form extends, which is of course not surprising, since this
is what the cut construction tells us.
\subsubsection{Collar parameter independent of~$\alpha$}
\label{subsubsection:collar2}
When we take $\tau=r-1$ as collar parameter, the function $g$ is
simply given by $r=1+\tau$, so $\Phi$ takes the form
\begin{equation}
\label{eqn:Phi}
\Phi\co (b;\rho,\vartheta) \longmapsto
\left\{\begin{array}{l}
s=\vartheta;\\
r=1-\rho^2,\\
\theta=b-h\vartheta.\end{array}\right.
\end{equation}
It follows that
\begin{equation}
\label{eqn:Phi*}
\Phi^*\alpha=(1-\rho^2)^2\,\rmd b+
\bigl(H_s\circ\Phi-h(1-\rho^2)^2\bigr)\,\rmd\vartheta.
\end{equation}
Now the extension problem is located in the second summand,
and the dependence on $H_s$ is more transparent than with the choice
made in Section~\ref{subsubsection:collar1}, where this dependence
is hidden in the function~$g$.
\subsection{Ellipsoids}
\label{subsection:ellipsoids}
Consider the Hamiltonian function $H(r\rme^{\rmi\theta})=a_2r^2+a_0$
with $a_0,a_2\in\R$, $a_0>0$, and $a_0+a_2=h\in\N$.
This satisfies the contact condition~(\ref{eqn:contact-alt}).
The function $H$ defines the Hamiltonian vector field
$X=-a_2\partial_{\theta}$, and the
Reeb vector field of the contact form $\alpha=H\,\rmd s+\lambda$
is $R_{\alpha}=(\partial_s+X)/a_0$.
We compute
\[ \Phi^*\alpha=(1-\rho^2)^2\,\rmd b+a_0(2-\rho^2)\rho^2\,\rmd\vartheta;\]
this formula also defines the extension of $\Phi^*\alpha$
as a contact form $\halpha$ over $\rho=0$. The Reeb vector field
of $\halpha$ is $R_{\halpha}=\partial_b+\partial_{\vartheta}/a_0$.
In cartesian coordinates $u+\rmi v=\rho\rme^{\rmi\vartheta}$
we have $\partial_{\vartheta}=
u\partial_v-v\partial_u$, so along the binding $B$ the Reeb vector
field equals~$\partial_b$.
Notice that if we fix $a_0$ and allow $a_2$ to vary (by integers),
the dynamics around the periodic Reeb orbit corresponding to
the fixed point $0\in D^2$ changes, while the one around the
periodic orbit $B$ does not.

This example gives the intrinsic description of the Reeb flow
on ellipsoids in~$\R^4$. Apart from the two periodic orbits just
mentioned, we have a foliation by $2$-tori, which in turn are linearly
foliated by Reeb orbits. Depending on $a_0$ being rational or not,
the Reeb orbits on these tori are periodic or dense.

Indeed, we can adapt the quotient map in the proof of Lemma~\ref{lem:quotient}
to this example. Consider the ellipsoid
\[ E_{a_0}:=\Bigl\{(z_1,z_2)\in\C^2\co
\frac{|z_1|^2}{a_0}+|z_2|^2=1\Bigr\}.\]
The quotient map
\[ \begin{array}{rccc}
\Psi\co & S^1\times D^2 & \longrightarrow & E_{a_0}\\
        & (s;r,\theta)  & \longmapsto     & \bigl(\sqrt{a_0(1-r^2)}
                                            \,\rme^{\rmi s},
                                            r\rme^{\rmi(\theta+hs)}\bigr)
\end{array}\]
pulls back the standard contact form $r_1^2\,\rmd\theta_1+r_2^2\,\rmd\theta_2$
on $E_{a_0}$ to
\[ \Psi^*(r_1^2\,\rmd\theta_1+r_2^2\,\rmd\theta_2)=
(a_0+a_2r^2)\,\rmd s+r^2\,\rmd\theta=\alpha,\]
and $T\Psi(R_{\alpha})=\partial_{\theta_1}/a_0+\partial_{\theta_2}$.
\subsection{The extension problem}
\label{subsection:extension}
We now return to the irrational pseudorotations of Fayad--Katok. Thus,
from now on the Hamiltonian functions $H_s^{\nu}$, $\nu\in\N$,
and their $C^{\infty}$-limit $H_s^{\infty}$ are those corresponding
to an irrational pseudorotation, as found in
Section~\ref{subsection:pseudo-hamil}.

We choose to work with a fixed collar parameter as in
Section~\ref{subsubsection:collar2}. Then the question
whether $\Phi^*\alpha_{\infty}$ extends as
a smooth $1$-form to the binding $B$ reduces to the following statement.

\begin{prop}
\label{prop:smooth}
The function  $f\co B\times\Int(D^2_{\sqrt{\varepsilon}})\rightarrow\R$
defined by
\begin{equation}
\label{eqn:f}
f(b,\rho\rme^{\rmi\vartheta}):=
\begin{cases}
\bigl(H_s^{\infty}\circ\Phi-h(1-\rho^2)^2\bigr)/\rho^2
       & \text{for $\rho\neq 0$},\\
2(h+a) & \text{for $\rho=0$},
\end{cases}
\end{equation}
where $a=\lim_{\nu\rightarrow\infty}p_{\nu}/q_{\nu}$, is smooth.
\end{prop}

The $1$-form $\Phi^*\alpha_{\infty}$ then extends smoothly over $\rho=0$ as
\[ \halpha:=(1-\rho^2)^2\,\rmd b+f\rho^2\,\rmd\vartheta.\]

The functions $f^{\nu}$, defined as in Proposition~\ref{prop:smooth},
with $H_s^{\infty}$ replaced by $H_s^{\nu}$ and $a$ by $p_{\nu}/q_{\nu}$,
are easily shown to be smooth, see Lemma~\ref{lem:smooth} below.

\begin{lem}
\label{lem:h+a}
When $h\in\N$ in the above construction is chosen such that $h+a>0$,
the extended $1$-form $\halpha$ is a contact form,
and $B\times\{0\}$ is a (positively oriented) Reeb orbit.
\end{lem}

\begin{proof}
The contact condition needs to be verified along $B\times\{0\}$.
We have
\[ \rmd\halpha=-4(1-\rho^2)\rho\,\rmd\rho\wedge\rmd b
+2f\rho\,\rmd\rho\wedge\rmd\vartheta +\rmd f\wedge\rho^2\,\rmd\vartheta,\]
and hence
\[ \halpha\wedge\rmd\halpha|_{\rho=0}=2f\,\rmd b\wedge
\rho\,\rmd\rho\wedge\rmd\vartheta>0,\]
provided that $f|_{B\times\{0\}}>0$.
Moreover, we have $\halpha(\partial_b)|_{\rho=0}=1$
and $i_{\partial_b}\rmd\halpha|_{\rho=0}=0$, so
$\partial_b$ is the Reeb vector field of $\halpha$
along $B\times\{0\}$.
\end{proof}

\begin{rem}
The condition $h+p_{\nu}/q_{\nu}>0$ is precisely the contact
condition (\ref{eqn:contact-alt}) for the $1$-form $R^{\nu}\,\rmd s+\lambda$,
where $R^{\nu}$ is the standard quadratic Hamiltonian
in~(\ref{eqn:R-hamil}). So the condition $h+a$ in the lemma
is simply saying that the strict inequality should also hold
in the limit $\nu\rightarrow\infty$.
\end{rem}

Thus, in order to demonstrate Theorem~\ref{thm:main},
it only remains to prove Proposition~\ref{prop:smooth}. The further
statements in Theorem~\ref{thm:main}, apart from the
dynamical convexity, then follow as in the proof
of Proposition~\ref{prop:intro} in Section~\ref{section:cuts}.

The embedding $\Phi$ in (\ref{eqn:Phi}) extends to a smooth map
\[ \begin{array}{rccc}
\tPhi\co & B\times[0,\sqrt{\varepsilon}\,)\times S^1
    & \longrightarrow & S^1\times D^2\\
         & (b,\rho,\vartheta)
    & \longmapsto     & \bigl(\vartheta,(1-\rho^2)
                        \rme^{\rmi(b-h\vartheta)}\bigr).
\end{array} \]
The function
\[ \begin{array}{rccc}
\tf\co & B\times[0,\sqrt{\varepsilon}\,)\times S^1
                   & \longrightarrow & \R\\
       & (b,\rho,\vartheta)
                   & \longmapsto & f(b,\rho\rme^{\rmi\vartheta})
\end{array} \]
lifting $f$ from (\ref{eqn:f}) can then be written as
\begin{equation}
\label{eqn:tf}
\tf(b,\rho,\vartheta)=\begin{cases}
\bigl(H_s^{\infty}\circ\tPhi-h(1-\rho^2)^2\bigr)/\rho^2
       & \text{for $\rho\neq 0$},\\
2(h+a) & \text{for $\rho=0$}.
\end{cases}
\end{equation}
Similarly, we have functions $\tf^{\nu}$, $\nu\in\N$, when we replace
$H_s^{\infty}$ by $H_s^{\nu}$ and $a$ by $p_{\nu}/q_{\nu}$
in the definition of~$\tf$.

\begin{lem}
\label{lem:smooth}
The function $\tf,\tf^{\nu}$ on
$B\times[0,\sqrt{\varepsilon}\,)\times S^1$ are smooth.
\end{lem}

\begin{proof}
By equations (\ref{eqn:R-hamil}) and (\ref{eqn:H}), we have
\[ H_s^{\nu}=R_s^{\nu}=h+\frac{p_{\nu}}{q_{\nu}}-
\frac{p_{\nu}}{q_{\nu}}r^2\;\;\;\text{near $\partial D^2$}.\]
It follows that
\[ H_s^{\nu}\circ\tPhi-h(1-\rho^2)^2=\Bigl(h+\frac{p_{\nu}}{q_{\nu}}\Bigr)
\cdot(2\rho^2-\rho^4)\]
for $\rho$ near and including~$0$. This shows that the
$\tf^{\nu}$ are smooth, and so are the~$f^{\nu}$.

Since $H_s^{\infty}$ is the $C^{\infty}$-limit of the
$H_s^{\nu}$, the function
\[ H_s^{\infty}\circ\tPhi-h(1-\rho^2)^2\]
vanishes to second order in $\rho$ at $\rho=0$,
and its second partial derivative with respect to $\rho$ at
$\rho=0$ equals $4(h+a)$. By a well-known lemma
of Morse~\cite[p.~349]{mors25}, cf.\ \cite[Lemma~1.2.3]{wall16},
this means that $\tf$ is smooth.
\end{proof}

In \cite{giro}, a function $u\co D^2_{\delta}\rightarrow\R$ having the
property that the lifted function $\tu\co[0,\delta]\times S^1\rightarrow\R$
is smooth is called \emph{weakly smooth}.
\subsection{$C^1$-functions in polar coordinates}
We now discuss the general question under which conditions
a $C^1$-function $\tu$ on $[0,\delta]\times S^1$ descends to
a $C^1$-function $u$ on $D^2_{\delta}$ when $(\rho,\vartheta)
\in[0,\delta]\times S^1$ are interpreted as polar coordinates.
We write the partial derivatives of $\tu$ as $\tu_{\rho}$ and
$\tu_{\vartheta}$, respectively.

\begin{lem}
\label{lem:C1}
Let $\tu\co[0,\delta]\times S^1\rightarrow\R$ be a $C^1$-function
with $\tu(0,\vartheta)$ independent of~$\vartheta$. Define
\[ \begin{array}{rccc}
u\co & D^2_{\delta}             & \longrightarrow & \R\\[.5mm]
     & \rho\rme^{\rmi\vartheta} & \longmapsto     & \tu(\rho,\vartheta).
\end{array}\]
Then $u$ is a $C^1$-function if and only if
\[ \tu_{\rho}(0,0)=-\tu_{\rho}(0,\pi),\;\;\;
\tu_{\rho}(0,\nicefrac{\pi}{2})=-\tu_{\rho}(0,\nicefrac{3\pi}{2}),\]
as well as
\[ \lim_{\rho\rightarrow 0}\Bigl(\cos\vartheta\,\tu_{\rho}-
\frac{\sin\vartheta}{\rho}\,\tu_{\vartheta}\Bigr)=\tu_{\rho}(0,0),\]
and
\[ \lim_{\rho\rightarrow 0}\Bigl(\sin\vartheta\,\tu_{\rho}+
\frac{\cos\vartheta}{\rho}\,\tu_{\vartheta}\Bigr)=
\tu_{\rho}(0,\nicefrac{\pi}{2}).\]
Here the limits $\lim_{\rho\rightarrow 0}$ are to be read as
$\lim_{m\rightarrow\infty}$ for any sequence $(\rho_m,\vartheta_m)$
with $\rho_m\rightarrow 0$; the sequence $(\vartheta_m)_{m\in\N}$
need not converge.
\end{lem}

\begin{proof}
In cartesian coordinates $z=x+\rmi y$ on $D^2_{\delta}$ we have
$\rho=\sqrt{x^2+y^2}$ and $\vartheta=\arctan(y/x)$.
It follows that, for $z\neq 0$,
\[ u_x=\cos\vartheta\,\tu_{\rho}-\frac{\sin\vartheta}{\rho}\,\tu_{\vartheta}\]
and
\[ u_y=\sin\vartheta\,\tu_{\rho}+\frac{\cos\vartheta}{\rho}\,\tu_{\vartheta}.\]
In $z=0$, we have
\[ u_x(0)=\lim_{t\rightarrow 0}
\frac{u(t)-u(0)}{t},\]
which, depending on the sign of $t\in\R\setminus\{0\}$, gives the limit
\[ \lim_{\substack{t\rightarrow 0\\ t>0}}
\frac{\tu(t,0)-\tu(0,0)}{t}=\tu_{\rho}(0,0),\]
or
\[ \lim_{\substack{t\rightarrow 0\\ t<0}}
\frac{\tu(-t,\pi)-\tu(0,\pi)}{t}=-\tu_{\rho}(0,\pi).\]
For the partial derivative $u_y$, the computations are analogous.
The lemma follows.
\end{proof}

When we verify the conditions of Lemma~\ref{lem:C1}
in the application to proving Proposition~\ref{prop:smooth},
we compute the limit $\lim_{\rho\rightarrow 0}$ as
a double limit $\lim_{m,n\rightarrow\infty}$ for a
sequence $(\rho_m,\vartheta_n)$ with $\rho_m\rightarrow 0$
and $\vartheta_n$ arbitrary.
There we shall need the following elementary lemma.

\begin{lem}
\label{lem:double}
Let $(a_{mn})_{m,n\in\N}$ be a double sequence of real numbers.
Suppose the following conditions are satisfied.
\begin{itemize}
\item[(i)] For $m\rightarrow\infty$, each of the sequences $(a_{mn})_{m\in\N}$
converges to some real number~$a_n$, uniformly in~$n$.
\item[(ii)] The limit $\lim_{n\rightarrow\infty}a_n=:a$ exists.
\end{itemize}
Then the limit $\lim_{m,n\rightarrow\infty}a_{m,n}$ exists and
equals~$a$.
\end{lem}

\begin{proof}
Uniform convergence in $n$ of the sequences $(a_{mn})_{m\in\N}$
means that for any $\varepsilon>0$ there is an $M(\varepsilon)\in\N$
such that
\[ |a_{mn}-a_n|<\varepsilon\;\;\;\text{for all $m\geq M(\varepsilon)$
and $n\in\N$.}\]
Convergence of $(a_n)_{n\in\N}$ means that there is an $N(\varepsilon)\in\N$
such that
\[ |a_n-a|<\varepsilon\;\;\;\text{for all $n\geq N(\varepsilon)$.}\]
Hence, for $m,n\geq\max\{M(\varepsilon),N(\varepsilon)\}$ we have
\[ |a_{mn}-a|\leq|a_{mn}-a_n|+|a_n-a|<2\varepsilon,\]
which proves the lemma. 
\end{proof}
\subsection{Proof of Proposition~\ref{prop:smooth} -- The first derivative}
We now apply Lem\-ma~\ref{lem:C1} to the function $\tf$
in (\ref{eqn:tf}) corresponding to the function $f$ in~(\ref{eqn:f}).
We suppress the $b$-coordinate, which is irrelevant for the
argument.

\begin{lem}
\label{lem:rho=0}
The function $\tf$ in (\ref{eqn:tf}) satisfies
\begin{equation}
\tf_{\rho}|_{\rho=0}=0\;\;\;\text{and}\;\;\;
\tf_{\rho\rho}|_{\rho=0}=-2(h+a).
\end{equation}
All other higher or mixed derivatives of $\tf$ vanish at $\rho=0$.
\end{lem}

\begin{proof}
From the proof of Lemma~\ref{lem:smooth} we have
\[ \tf^{\nu}=\Bigl(h+\frac{p_{\nu}}{q_{\nu}}\Bigr)\cdot(2-\rho^2)\]
near $\rho=0$. Since $\tf$ is the $C^{\infty}$-limit of the $\tf^{\nu}$,
the lemma follows.
\end{proof}

Let $(\rho_n,\vartheta_n)_{n\in\N}$ be a sequence
in $(0,\delta]\times S^1$ with $\rho_n\rightarrow 0$ for $n\rightarrow\infty$.
We need to verify that
\begin{equation}
\label{eqn:lim1}
\tf_{\rho}(\rho_n,\vartheta_n)\longrightarrow 0
\end{equation}
and
\begin{equation}
\label{eqn:lim2}
\frac{1}{\rho_n}\tf_{\vartheta}(\rho_n,\vartheta_n)\longrightarrow 0.
\end{equation}

For the limit in (\ref{eqn:lim1}), set
\[ a_{mn}:=\tf_{\rho}(\rho_m,\vartheta_n)
=\rho_m\cdot\frac{\tf_{\rho}(\rho_m,\vartheta_n)}{\rho_m}.\]
The limit $\lim_{m\rightarrow\infty}a_{mn}$
is uniform in $n$ (and equals~$0$) thanks to the following lemma.
With Lemma~\ref{lem:double} we then conclude
$\lim_{m,n\rightarrow\infty}a_{mn}=0$.

\begin{lem}
\label{lem:uniform}
For $\rho\rightarrow 0$, the difference quotient
$\tf_{\rho}(\rho,\vartheta)/\rho$ converges to the derivative
$\tf_{\rho\rho}(0,\vartheta)=-2(h+a)$ uniformly in~$\vartheta$.
\end{lem}

\begin{proof}
We make the following estimates with the mean value theorem:
\begin{eqnarray*}
|\tf_{\rho}(\rho_0,\vartheta_1)-\tf_{\rho}(\rho_0,\vartheta_0)|
 & \leq & \max_{\vartheta}|\tf_{\rho\vartheta}(\rho_0,\vartheta)|\cdot
          |\vartheta_1-\vartheta_0|\\
 & \leq & \max_{\rho,\vartheta}|\tf_{\rho\vartheta\rho}(\rho,\vartheta)|\cdot
          \rho_0\cdot|\vartheta_1-\vartheta_0|.
\end{eqnarray*}
Here $|\vartheta_1-\vartheta_0|$ denotes the length of a circular arc
between $\vartheta_0$ and $\vartheta_1$; the maximum is taken over
$\vartheta\in S^1$ and $\rho\in[0,\rho_0]$.

We then estimate
\begin{eqnarray*}
\Bigl|\frac{\tf_{\rho}(\rho_0,\vartheta_1)}{\rho_0}+2(h+a)\Bigr|
 & \leq & \Bigl|\frac{\tf_{\rho}(\rho_0,\vartheta_1)}{\rho_0}
          -\frac{\tf_{\rho}(\rho_0,\vartheta_0)}{\rho_0}\Bigr|
          +\Bigl|\frac{\tf_{\rho}(\rho_0,\vartheta_0)}{\rho_0}
          +2(h+a)\Bigr|\\
 & \leq & \max_{\rho,\vartheta}|\tf_{\rho\vartheta\rho}(\rho,\vartheta)|\cdot
          |\vartheta_1-\vartheta_0|
          +\Bigl|\frac{\tf_{\rho}(\rho_0,\vartheta_0)}{\rho_0}
          +2(h+a)\Bigr|,
\end{eqnarray*}
which, together with the compactness of~$S^1$,
gives the desired uniformity in~$\vartheta$.
\end{proof}

For the limit in (\ref{eqn:lim2}), one applies completely analogous
arguments to the double sequence
$\tf_{\vartheta}(\rho_m,\vartheta_n)/\rho_m$.

This shows that the function $f$ in Proposition~\ref{prop:smooth} is
continuously differentiable.
\subsection{Proof of Proposition~\ref{prop:smooth} -- Higher derivatives}
\label{subsection:higher}
In principle, higher deri\-va\-tives one can deal with by iterating
Lemma~\ref{lem:C1}. In order to establish that $f$ is $C^2$,
we write out explicitly the second derivative~$f_{xx}$.
For $f_{xy}$ and $f_{yy}$ the considerations are analogous.

In $z\neq 0$ we have
\begin{eqnarray}
\label{eqn:fxx}
f_{xx} & = & \tf_{\rho\rho}\cos^2\vartheta
             -\tf_{\rho\vartheta}\frac{2\sin\vartheta\cos\vartheta}{\rho}
             +\tf_{\rho}\frac{\sin^2\vartheta}{\rho}\\
       &   & \mbox{}+\tf_{\vartheta\vartheta}\frac{\sin^2\vartheta}{\rho^2}
             +\tf_{\vartheta}\frac{2\sin\vartheta\cos\vartheta}{\rho^2}.
             \nonumber
\end{eqnarray}
In $z=0$, we find
\[  f_{xx}(0)=\tf_{\rho\rho}(0,0)=\tf_{\rho\rho}(0,\pi).\]

Recall the properties
of $\tf$ stated in Lemma~\ref{lem:rho=0}.
The derivative $f_{xx}(0)$
exists thanks to $\tf_{\rho\rho}(0,0)$ and $\tf_{\rho\rho}(0,\pi)$
both being equal to $-2(h+a)$. For the continuity of $f_{xx}$
in $z=0$, we consider the summands on the right-hand side of
(\ref{eqn:fxx}) in turn. We evaluate these summands at
a point $(\rho_m,\vartheta_n)$, and consider the limit $m\rightarrow\infty$,
assuming that $\rho_m\rightarrow 0$ in this limit.

\begin{itemize}
\item[(i)] The term $\tf_{\rho\rho}\cos^2\vartheta_n$ converges to
\[ \tf_{\rho\rho}|_{\rho=0}\cos^2\vartheta_n=-2(h+a)\cos^2\vartheta_n.\]

\item[(ii)] The term
$-2\tf_{\rho\vartheta}\sin\vartheta_n\cos\vartheta_n/\rho$
converges to zero, since $\tf_{\rho\vartheta\rho}|_{\rho=0}=0$.

\item[(iii)] The term $\tf_{\rho}\sin^2\vartheta_n/\rho$ converges to
\[ \tf_{\rho\rho}|_{\rho=0}\sin^2\vartheta_n=-2(h+a)\sin^2\vartheta_n.\]

\item[(iv)] The term $\tf_{\vartheta\vartheta}\sin^2\vartheta_n/\rho^2$
is seen to converge to zero by applying l'H\^opital's rule twice, since
$\tf_{\vartheta\vartheta\rho\rho}|_{\rho=0}=0$.

\item[(v)] The limit of the term
$2\tf_{\vartheta}\sin\vartheta_n\cos\vartheta_n/\rho^2$ equals zero by the
same argument as in~(iv).
\end{itemize}

The uniformity of these limits can be seen by the same reasoning as above.
Thus, we have
\[ \lim_{m,n\rightarrow\infty}f_{xx}(\rho_m,\vartheta_n)=-2(h+a)=f_{xx}(0).\]

This argument, applied analogously to $f_{xy}$ and~$f_{yy}$, shows
that the function $f$ in Proposition~\ref{prop:smooth} is~$C^2$.

In order to establish that the function $f$
in question is $C^{\infty}$ near $z=0$, we need a more
systematic approach. We shall describe one such approach that
is general enough to apply to the Fayad--Katok examples.

\begin{lem}
\label{lem:u}
Let $\tu\co[0,\delta]\times S^1\rightarrow\R$ be a smooth function with
$\tu(0,\vartheta)$ independent of~$\vartheta$, and let
$u(x+\rmi y)=u(\rho\rme^{\rmi\vartheta})=\tu(\rho,\vartheta)$ be the induced
function on~$D^2_{\delta}$. For $k\in\N$ and $j\in\{0,1,\ldots,k\}$,
the partial derivatives
\[ \frac{\partial^ku}{\partial x^j\,\partial y^{k-j}}\]
in $z\neq 0$ are sums of terms
\[ \frac{\partial^{\ell}\tu}{\partial\rho^i\,\partial\vartheta^{\ell-i}}
\cdot\frac{a_{ji}^{k\ell}(\vartheta)}{\rho^{k-i}},\]
where $1\leq\ell\leq k$, $0\leq i\leq\ell$, and $a_{ji}^{k\ell}$ is a
polynomial in $\sin\vartheta$ and $\cos\vartheta$.
\end{lem}

\begin{proof}
For $k=1$ this is confirmed by the formulae for $u_x$ and $u_y$
in the proof of Lemma~\ref{lem:C1}. Then argue by induction over~$k$,
using the fact that $\partial\rho/\partial x=\cos\vartheta$
and $\partial\vartheta/\partial x=-\sin\vartheta/\rho$, and similar
expressions for the derivatives with respect to~$y$.
\end{proof}

\begin{lem}
Let $\tu\co [0,\delta]\times S^1\rightarrow\R$ be a smooth function satisfying
\[ \frac{\partial^{\ell}\tu}{\partial\rho^i\,\partial\vartheta^{\ell-i}}
(0,\vartheta)=0\]
for all $\ell\in\N_0$, $0\leq i\leq\ell$ and $\vartheta\in S^1$.
Then the function $u\co D^2_{\delta}\rightarrow\R$,
defined as in Lemma~\ref{lem:u}, is smooth and vanishes to
infinite order at~$z=0$.
\end{lem}

\begin{proof}
Again we argue by induction on the order of derivatives. The limit
conditions in Lemma~\ref{lem:C1} are satisfied by l'H\^opital's rule,
so the function $u$ is~$C^1$.

By a repeated application of l'H\^opital's rule we also
see that the terms from the preceding lemma satisfy
\[ \lim_{\rho\rightarrow 0}
\frac{\partial^{\ell}\tu}{\partial\rho^i\,\partial\vartheta^{\ell-i}}
\cdot\frac{a_{ji}^{k\ell}(\vartheta)}{\rho^{k-i}}=0.\]
These limits are uniform in $\vartheta$ by arguments
as in the proof of Lemma~\ref{lem:uniform}.
It follows that
\[ \lim_{z\rightarrow 0}\frac{\partial^ku}{\partial x^j\,\partial y^{k-j}}
(z)=0.\]

For the inductive step, we need to show that $u$ is of class $C^{k+1}$,
presuming that we have already established it to be of class~$C^k$,
with vanishing partial derivatives at~$z=0$. Thus, let $v$ be
a $k^{\mathrm{th}}$ partial derivative of~$u$, and $\tv$ its lift
to $[0,\delta]\times S^1$. Then, as in Lemma~\ref{lem:C1},
\[ v_x(0)=\lim_{t\rightarrow 0}\frac{v(t)}{t}=
\left\{\begin{array}{rl}
\tv_{\rho}(0,0)=0    & \text{for $t>0$},\\
-\tv_{\rho}(0,\pi)=0 & \text{for $t<0$,}
\end{array}\right.\]
so the derivative $v_x(0)$ exists. The continuity of $v_x(z)$
at $z=0$ follows from the limit behaviour of the derivatives
described above. For the derivative $v_y$ the argument is
analogous.
\end{proof}

Proposition~\ref{prop:smooth} now follows by applying this lemma
to the functions
\[ u=f-(h+a)(2-x^2-y^2)\;\;\;\text{and}\;\;\;
\tu=\tf-(h+a)(2-\rho^2).\]

This concludes the proof of Theorem~\ref{thm:main}, except for the
dynamical convexity, which will be established in
Section~\ref{subsubsection:convex}.
\subsection{A more general sufficient criterion}
\label{subsection:sufficient}
We continue to write $\Phi$ for the embedding (\ref{eqn:Phi}).

\begin{thm}
\label{thm:main-general}
Let $\psi\co D^2\rightarrow D^2$ be an area-preserving diffeomorphism
generated by a $2\pi$-periodic Hamiltonian function~$H_s$
with $H_s|_{\partial(S^1\times D^2)}\equiv h\in\N$.
Consider the function
\[ f(b,\rho\rme^{\rmi\vartheta}):=
\bigl(H_s\circ\Phi-h(1-\rho^2)^2\bigr)/\rho^2\]
on $B\times\bigl(\Int(D^2_{\sqrt{\varepsilon}})\setminus\{0\}\bigr)$.
If $f$ extends continuously over $B\times\{0\}$, and if the lifted
function $\tf$ on $B\times[0,\sqrt{\varepsilon})\times S^1$
is smooth and has the $\infty$-jet, along $B\times\{0\}\times S^1$, of
the lift of a smooth function, then $\psi$ embeds into
a Reeb flow on~$S^3$.
\end{thm}

\begin{proof}
By the discussion in Section~\ref{subsection:extension},
a sufficient condition for $\psi$ to embed into a Reeb flow on $S^3$
is that the function $f$
extends smoothly as a positive function over $B\times\{0\}$.
By the analysis in the preceding section, this in turn is
equivalent to the conditions on $f$ stated in the theorem.
\end{proof}

The condition of the theorem is satisfied, as one ought to expect, when
$H_s$ satisfies Assumption~\ref{ass:H}. More generally, it suffices to
assume, for instance, that the $\infty$-jet of $H_s$ along
$\partial (S^1\times D^2$) is that of a function depending only on~$r$.
\subsection{Conjugation invariance}
In our proof of Theorem~\ref{thm:main} and the more general statement in
the preceding section, we have relied on an explicit coordinate description
of $S^3$ as a contact cut of $S^1\times D^2$. In this section we want
to discuss the conjugation invariance of the construction,
which amounts to saying that the specific coordinates are irrelevant.

There is one version of conjugation invariance that is completely
tautological. Let
\[ \varphi\co D^2\stackrel{\cong}{\longrightarrow}
\{0\}\times D^2\subset S^1\times D^2\]
be any embedding of $D^2$ with image $\{0\}\times D^2$, and assume that
\[ \psi\co\{0\}\times D^2\rightarrow\{0\}\times D^2\]
is the return map of a given Reeb flow. Then, with respect
to the embedding
\[ D^2\longrightarrow S^3=(S^1\times D^2)/\!\!\sim\]
given by $\varphi$, the return map is $\varphi^{-1}\circ\psi\circ\varphi$,
which preserves the area form $\varphi^*\omega$.

More restrictively, we may fix the disc $\{0\}\times D^2\subset
S^1\times D^2$ and ask whether it is possible to find a new contact
form on $S^1\times D^2$ whose return map on $\{0\}\times D^2$
is the conjugate of the previous one. The following proposition
gives the (still quite tautological) answer.

\begin{prop}
\label{prop:conjugate}
Let $\psi\co D^2\rightarrow D^2$ be an area-preserving diffeomorphism
generated by a $2\pi$-periodic Hamiltonian function $H_s$ with
$H_s|_{\partial (S^1\times D^2)}\equiv h\in\N$. Assume that
$\psi$ embeds into a Reeb flow on $S^3$ by the cut construction
described in Section~\ref{section:pseudorot}, starting from the
contact form $\alpha=H_s\,\rmd s+\lambda$. Let $\varphi
\co D^2\rightarrow D^2$ be a further area-preserving diffeomorphism.
Then the conjugate $\varphi^{-1}\circ\psi\circ\varphi$ likewise
embeds into a Reeb flow.
\end{prop}

\begin{proof}
The diffeomorphism $\varphi^{-1}\circ\psi\circ\varphi$ is generated
by the Hamiltonian function $H_s\circ\varphi$, see
\cite[Exercise~3.1.14]{mcsa17}. Regard $\varphi$ as
a diffeomorphism of $S^1\times D^2$. Then
\[ \varphi^*\alpha= (H_s\circ\varphi)\,\rmd s+\varphi^*\lambda \]
and
\[ \rmd(\varphi^*\alpha)=\rmd(H_s\circ\varphi)\wedge\rmd s+\omega,\]
since $\varphi$ is area-preserving. With $R=\partial_s+X_s$ as
before, where $X_s$ is the Hamiltonian vector field $X_{H_s}$ of~$H_s$,
it follows that $\varphi^*R=\partial_s+X_{H_s\circ\varphi}$.
Incidentally, this provides a quick solution to the cited exercise.

We take the cut of $S^1\times D^2$ with respect to the
$S^1$-action by $\varphi^*Y$ on the boundary, and we
replace the embedding $\Phi$ of a pointed neighbourhood of the binding by
the composition $\varphi^{-1}\circ\Phi$. The pull-back of $\varphi^*\alpha$
under this embedding equals $\Phi^*\alpha$, which extends by assumption.
\end{proof}
\subsection{The choice of primitive}
\label{subsection:dF}
In some sense, Proposition~\ref{prop:conjugate} is not entirely satisfactory,
since it really talks about the embeddability not of $\psi$, but of
the pair $(\psi,\lambda)$. It says that the embeddability of
$(\psi,\lambda)$ is equivalent to that of
\[ (\varphi^{-1}\circ\psi\circ\varphi, \varphi^*\lambda),\]
and this just amounts to a global change of coordinates.

We want to show that, at least up to
$C^2$-differentiability of the Hamiltonian function~$H_s$, the question
whether $\psi$ embeds is independent of the choice of primitive for the
area form~$\omega$.

A general primitive of $\omega$ is of the form $\lambda+\rmd F$,
with $F\co D^2\rightarrow\R$ a smooth function. We assume that,
possibly after adding a large integer to~$H_s$, the $1$-form
\[ \alpha_{F}:=H_s\,\rmd s+\lambda+\rmd F\]
is a contact form, that is, it satisfies (\ref{eqn:contact})
with $\lambda$ replaced by $\lambda+\rmd F$.
We want the new contact form $\alpha_F$ to be invariant under the
$S^1$-action on $\partial (S^1\times D^2)$
generated by $Y=\partial_s-h\partial_{\theta}$.
Near $r=1$ we have
\[ \rmd F=\frac{\partial F}{\partial r}\,\rmd r+
\frac{\partial F}{\partial\theta}\,\rmd\theta,\]
so the invariance requirement, cf.\ Lemma~\ref{lem:invariant}, becomes
\[ 0=L_{Y}\rmd F|_{r=1}=\rmd\bigl(\rmd F(\partial_{\theta})\bigr)|_{r=1}.\]
This means that
\begin{equation}
\label{eqn:partialF}
\frac{\partial^2F}{\partial\theta^2}\Big|_{r=1}=0\;\;\;\text{and}\;\;\;
\frac{\partial^2 F}{\partial r\,\partial\theta}\Big|_{r=1}=0.
\end{equation}
The first condition forces $(\partial F/\partial\theta)|_{r=1}=0$.
Notice that $\alpha_F(Y)|_{r=1}=0$, so $\alpha_F$ descends to
a $1$-form on the quotient $S^3=(S^1\times D^2)/\!\!\sim$.

With the embedding $\Phi$ from (\ref{eqn:Phi}), we have
\[ \Phi^*\rmd F=-\frac{\partial F}{\partial r}(1-\rho^2,b-h\vartheta)\,
\rho\,\rmd\rho+\frac{\partial F}{\partial\theta}(1-\rho^2,b-h\vartheta)\cdot
(\rmd b-h\,\rmd\vartheta).\]

In $\Phi^*\alpha_F$ there are no other terms in $\rmd\rho$, so if
we want $\Phi^*\alpha_F$ to extend as a contact form over $\rho=0$, the
least we need to require is that
\begin{equation}
\label{eqn:tf2}
\tf(b,\rho,\vartheta):=\frac{\partial F}{\partial r}
(1-\rho^2,b-h\vartheta),\;\;\;(b,\rho,\vartheta)\in
B\times [0,\delta]\times S^1,
\end{equation}
is the lift of a smooth function $f\co B\times D^2_{\delta}\rightarrow\R$.
As before, we are going to suppress the $b$-coordinate.

With this requirement understood, we have the following proposition,
which says that the remaining terms in $\Phi^*\rmd F$
extend as a $C^2$-form over $\rho=0$. So we do not, up to $C^2$, gain more
flexibility in the conditions on $H_s$ by adding $\rmd F$ to
the primitive~$\lambda$.

\begin{prop}
The function
\[ (\rho,\vartheta)\longmapsto\frac{\partial F}{\partial\theta}
(1-\rho^2,b-h\vartheta)\]
on $[0,\delta]\times S^1$ (for any fixed $b\in B$) equals $\rho^2$ times the
lift of a $C^2$-function $D^2_{\delta}\rightarrow\R$.
\end{prop}

\begin{proof}
By (\ref{eqn:partialF}) and the lemma
of Morse~\cite[Lemma~1.2.3]{wall16} used earlier,
we can write
\begin{equation}
\label{eqn:G}
\frac{\partial F}{\partial\theta}(r,\theta)=(r-1)^2G(r,\theta)
\end{equation}
with a smooth function~$G$. Hence,
\[ \frac{\partial F}{\partial\theta}(1-\rho^2,b-h\vartheta)=
\rho^4G(1-\rho^2,b-h\vartheta).\]
We therefore need to show that
\[ \tg(\rho,\vartheta):=\rho^2G(1-\rho^2,b-h\vartheta),\;\;\;
(\rho,\vartheta)\in[0,\delta]\times S^1,\]
is the lift of a $C^2$-function $g\co D^2_{\delta}\rightarrow\R$.

The derivative $\partial\tf/\partial\vartheta$, where
$\tf$ is the function defined in~(\ref{eqn:tf2}), is the
lift of the smooth function $\partial f/\partial\vartheta$.
On the other hand, we can write this derivative upstairs as
\[ \frac{\partial\tf}{\partial\vartheta}(\rho,\vartheta)=
-h\,\frac{\partial^2F}{\partial\theta\,\partial r}
(1-\rho^2,b-h\vartheta).\]
In other words, the function
\[ \tk(\rho,\vartheta):=\frac{\partial^2F}{\partial\theta\,\partial r}
(1-\rho^2,b-h\vartheta) \]
is the lift of a smooth function $k\co D^2_{\delta}\rightarrow\R$.
But, by (\ref{eqn:G}),
\begin{eqnarray*}
\tk(\rho,\vartheta)
 & = & \frac{\partial^2F}{\partial r\,\partial\theta}
       (1-\rho^2,b-h\vartheta)\\
 & = & -2\rho^2G(1-\rho^2,b-h\vartheta)+\rho^4\frac{\partial G}{\partial r}
       (1-\rho^2,b-h\vartheta)\\
 & = & -2\tg(\rho,\vartheta)+\rho^4\frac{\partial G}{\partial r}
       (1-\rho^2,b-h\vartheta).
\end{eqnarray*}
The second summand in this last expression is the lift of a $C^2$-function
$D^2_{\delta}\rightarrow\R$ by the considerations in
Lemma~\ref{lem:C1} and Section~\ref{subsection:higher}.
It follows that $\tg$ is the lift of a $C^2$-function, as we wanted to
show.
\end{proof}
\section{Dynamical invariants}
\label{section:invariants}
In this section we compute some invariants of the Reeb flows
on $S^3$ constructed via the cut construction, viz., the
Conley--Zehnder indices of the periodic Reeb orbits, and
the self-linking number of the binding orbit.

Throughout this section we assume, as before, that we are dealing
with a contact form on $S^3=(S^1\times D^2)/\!\!\sim$
coming from a contact form $\alpha=H_s\,\rmd s+\lambda$
on $S^1\times D^2$, where $H_s$ satisfies the boundary
condition~(\ref{eqn:boundary}), and the quotient is taken with
respect to the $S^1$-action on $\partial (S^1\times D^2)$
defined by the flow of $Y=\partial_s-h\partial_{\theta}$.
Moreover, it is of course assumed that $H_s$ has been
chosen such that $\Phi^*\alpha$ in (\ref{eqn:Phi*})
extends as a contact form over $\rho=0$; for instance,
one may assume the sufficient condition described in
Section~\ref{subsection:sufficient}.

Additionally, we impose the condition $H_s>0$. Notice that
adding a large natural number to the Hamiltonian
does not change the vector field $R=\partial_s+X_s$,
so this merely leads to a reparametrisation of the Reeb orbits.
This assumption on $H_s$ simplifies the discussion of framings.
\subsection{Framings}
In this section we describe trivialisations of the contact
plane fields over $\Int(S^1\times D^2)$ and near the
binding orbit $B=\bigl(\partial(S^1\times D^2)\bigr)/S^1$.
The comparison of these two framings will allow us to compute
the dynamical invariants.

Over $\Int(S^1\times D^2)$, the contact structure $\ker\alpha$
is trivialised by the oriented frame
\[ \left\{\begin{array}{rcl}
\bfe_1 & = & H_s\partial_x+y\partial_s,\\
\bfe_2 & = & H_s\partial_y-x\partial_s.
\end{array}\right.\]
Write $\Phi^*\alpha$ from (\ref{eqn:Phi*}) as
\begin{equation}
\label{eqn:Phi*f}
\Phi^*\alpha=(1-\rho^2)^2\,\rmd b+f(b,\rho\rme^{\rmi\vartheta})
\cdot(u\,\rmd v-v\,\rmd u),
\end{equation}
where $u+\rmi v=\rho\rme^{\rmi\vartheta}$. By assumption, $f$ extends
smoothly to $\rho=0$.
We then see that the contact structure $\ker(\Phi^*\alpha)$
is trivialised by the oriented frame
\[ \left\{\begin{array}{rcl}
\bfe_1' & = & (1-\rho^2)^2\partial_u+fv\partial_b,\\[.5mm]
\bfe_2' & = & (1-\rho^2)^2\partial_v-fu\partial_b.
\end{array}\right.\]
Away from $r=0$ we have
\[ \left\{\begin{array}{rcl}
\partial_x & = & \cos\theta\,\partial_r-
                 \frac{\sin\theta}{r}\,\partial_{\theta},\\[.5mm]
\partial_y & = & \sin\theta\,\partial_r+
                 \frac{\cos\theta}{r}\,\partial_{\theta}.
\end{array}\right.\]
There are analogous expressions for $\partial_u,\partial_v$,
with $(r,\theta)$ replaced by $(\rho,\vartheta)$.

The differential $T\Phi$ of $\Phi$ in (\ref{eqn:Phi}) is given by
\[ \left\{\begin{array}{rcl}
T\Phi(\partial_b)           & = & \partial_{\theta},\\
T\Phi(\partial_{\rho})      & = & -2\sqrt{1-r}\,\partial_r,\\
T\Phi(\partial_{\vartheta}) & = & \partial_s-h\partial_{\theta}.
\end{array}\right.\]
It follows that
\begin{eqnarray*}
T\Phi(\bfe_1') & = & -r^2\Bigl(2\sqrt{1-r}\,\cos s\,\partial_r+
                     \frac{\sin s}{\sqrt{1-r}}\,
                     (\partial_s-h\partial_{\theta})\Bigr)\\
               &   & \mbox{}+f(\theta+hs,\sqrt{1-r}\,\rme^{\rmi s})
                     \sqrt{1-r}\,\sin s\,\partial_{\theta}
\end{eqnarray*}
and
\begin{eqnarray*}
T\Phi(\bfe_2') & = & -r^2\Bigl(2\sqrt{1-r}\,\sin s\,\partial_r-
                     \frac{\cos s}{\sqrt{1-r}}\,
                     (\partial_s-h\partial_{\theta})\Bigr)\\
               &   & \mbox{}-f(\theta+hs,\sqrt{1-r}\,\rme^{\rmi s})
                     \sqrt{1-r}\,\cos s\,\partial_{\theta}.
\end{eqnarray*}

One particular case of interest will be when
$C:=S^1\times\{0\}\subset S^1\times D^2$ is a periodic Reeb orbit.
Observe that the annulus
\[ \bigl\{(s,r\rme^{-\rmi hs})\co s\in S^1,\, 0\leq r\leq 1\bigr\}
\subset S^1\times D^2\]
descends to a disc $\Delta$ in $S^3=(S^1\times D^2)/\!\!\sim$ with
boundary $\partial\Delta=S^1\times\{0\}$. Along the Reeb orbit $C$,
the surface framing given by $\Delta$ defines a trivialisation
of $\ker\alpha|_{S^1\times\{0\}}$. In a neighbourhood
of $C$, the contact planes project isomorphically
onto the tangent planes to the $D^2$-factor. With respect to this
projection, the oriented frame of the contact structure defined by $\Delta$
is then given by $(\partial_r,\partial_{\theta})$ in a pointed
neighbourhood of $S^1\times\{0\}$.
\subsection{Conley--Zehnder indices}
\label{subsection:muCZ}
We can now compute the Conley--Zehnder indices $\muCZ$ of the periodic Reeb
orbits in some examples. Recall that a contact form on the $3$-sphere
is called \emph{dynamically convex} if every periodic Reeb orbit
has index $\muCZ\geq 3$ \cite[Definition~1.2]{hwz98}.
\subsubsection{Irrational ellipsoids}
We begin with the ellipsoids from Section~\ref{subsection:ellipsoids},
that is, we consider the Hamiltonian function $H(r\rme^{\rmi\theta})=
a_2r^2+a_0$. The condition $a_0>0$ is equivalent to the
contact condition~(\ref{eqn:contact-alt}). The condition $a_0+a_2=h\in\N$
guarantees that $H>0$. We assume $a_0\in\R^+\setminus\Q$. Then there
are precisely two periodic Reeb orbits: the binding orbit~$B$,
and the central orbit $C=S^1\times\{0\}\subset S^1\times D^2$.

The meridian $\{0\}\times \partial D^2$ of $S^1\times D^2$ may be
taken as a representative of the binding orbit~$B$, so
$B$ bounds the disc $\{0\}\times D^2$ in $S^1\times D^2$.
The vector field $T\Phi(\bfe_1')$, for $r$ near but different from~$1$,
and for $s=0$, takes the form
\[ T\Phi(\bfe_1')=-2r^2\sqrt{1-r}\,\partial_r.\]
This makes one positive twist with respect to the frame $(\bfe_1,\bfe_2)$
as we go once along the meridian.

Near $B$ the Reeb vector field equals $\partial_b+\partial_{\vartheta}/a_0$,
so as we go once along $B$, the Reeb flow defines a rotation through
an angle $2\pi/a_0$ with respect to the frame $(\bfe_1',\bfe_2')$.
Thus, with respect to the frame $(\bfe_1,\bfe_2)$, we have a rotation through
an angle $2\pi(1+\nicefrac{1}{a_0})$.

By the definition of the Conley--Zehnder index~$\muCZ$,
see \cite[Section~8.1]{hwz03} or \cite[Section~2.2]{hls15},
we have $\muCZ(B)=2n+1$, where $n\in\N$ is the natural number
determined by $1+\nicefrac{1}{a_0}\in (n,n+1)$.

Near $C$ the Reeb vector field equals $(\partial_s-a_2\partial_{\theta})/a_0$.
The normalisation with return time $2\pi$ is $\partial_s-a_2\partial_{\theta}$.
Thus, with respect to the frame $(\bfe_1,\bfe_2)$ we make $-a_2$ twists
as we go once along the central orbit. The frame defined by the
disc $\Delta$ makes $-h$ twists relative to $(\bfe_1,\bfe_2)$. It follows that
the Reeb flow rotates through an angle $2\pi(h-a_2)=2\pi a_0$
with respect to the surface framing. Finally, the frame of $\ker\alpha$
that extends over $\Delta$ is the one defined by $(\bfe_1',\bfe_2')$ near the
centre of the disc. Up to positive factors, the projection
of $T\Phi(\bfe_1')$ onto the tangent planes to the $D^2$-factor
is of the form
\[ -\cos s\,\partial_r+\sin s\,\partial_{\theta}.\]
As we go once along $C$, this makes one negative twist with respect to
the frame $(\partial_r,\partial_{\theta})$. It follows that
the Reeb flow makes $a_0+1$ twists relative to the frame
$(\bfe_1',\bfe_2')$ (or its image under $T\Phi$). This gives
$\muCZ(C)=2m+1$, where $m\in\N$ is determined by
$a_0+1\in(m,m+1)$.

We see that, no matter what choice we make for $a_0\in\R^+\setminus\Q$,
one of the periodic orbits $B,C$ has $\muCZ=3$; the other,
$\muCZ=2n+1$ with $n\geq 2$. For an earlier proof of this
well-known result see~\cite[Lemma~1.6]{hwz95}.
\subsubsection{Irrational pseudorotations}
\label{subsubsection:convex}
As we have seen in the proof of Lemma~\ref{lem:smooth},
the Hamiltonian function describing an irrational pseudorotation arises
as the $C^{\infty}$-limit of Hamiltonians $H_s^{\nu}$ which
near $\partial D^2$ are given by
\[ H_s^{\nu}(r\rme^{\rmi\theta})=h+\frac{p_{\nu}}{q_{\nu}}-
\frac{p_{\nu}}{q_{\nu}}r^2.\]
In fact, the conjugating diffeomorphisms $\varphi_{\nu}$ in the
Fayad--Katok construction equal the identity map also near $0\in D^2$,
so there we have the same description of~$H_s^{\nu}$.
These Hamiltonians give rise to
functions $f^{\nu}$ in the description (\ref{eqn:Phi*f}) of $\Phi^*\alpha$
of the form
\[ f^{\nu}(b,\rho\rme^{\rmi\vartheta})=\Bigl(h+\frac{p_{\nu}}{q_{\nu}}\Bigr)
\cdot(2-\rho^2).\]
It follows that the $\infty$-jet of the limit Hamiltonian
$H_s^{\infty}$ along the central orbit $C$ equals that of
\[ H(r\rme^{\rmi\theta})=h+a-ar^2,\]
and the $\infty$-jet of the extended contact form along $B$ equals that of
\[ (1-\rho^2)^2\,\rmd b+(h+a)(2-\rho^2)\rho^2\,\rmd\vartheta.\]
This is precisely the situation of the irrational ellipsoids with
$a_0=h+a$ and $a_2=-a$. Recall from Lemma~\ref{lem:h+a} that
$h+a>0$. Summarising our arguments, we have the
following result, which completes
the proof of Theorem~\ref{thm:main}.

\begin{prop}
\label{prop:muCZ}
The irrational pseudorotations of Fayad--Katok embed into a Reeb flow
on $S^3$ whose periodic orbits $B,C$ have Conley--Zehnder indices
\[ \muCZ(B)=2n+1,\;\;\text{where $n\in\N$ is determined by}\;\;
1+\frac{1}{h+a}\in (n,n+1),\]
and
\[ \muCZ(C)=2m+1,\;\;\text{where $m\in\N$ is determined by}\;\;
h+a+1\in (m,m+1).\]
In particular, the contact form defining this Reeb flow
is dynamically convex.\qed
\end{prop}

Similar computations can be performed for general Hamiltonian
functions $H_s$ on $D^2$ that give rise to a contact form on~$S^3$.
The considerations above suggest that the
the Conley--Zehnder indices of periodic Reeb orbits corresponding
to periodic points of the diffeomorphism $\psi$ defined by $H_s$
can be determined from the local behaviour of $H_s$ near the
periodic point in question.

\begin{rem}
For Reeb flows on the $3$-sphere with two periodic orbits forming
a Hopf link, Hryniewicz--Momin--Salom\~ao~\cite[Theorem~1.2]{hms15}
describe a non-resonance condition that forces the existence
of infinitely many periodic orbits. This condition is
formulated in terms of the so-called transverse rotation number of the
two given periodic orbits. Our argument leading to
Proposition~\ref{prop:muCZ} shows that in the situation of that
proposition the transverse rotation numbers are given by
$\rho_0=1+1/(h+a)$ and $\rho_1=1+h+a$, respectively. The numbers
$\theta_i=\rho_i-1$ defined in \cite[Theorem~1.2]{hms15}
then become $\theta_0=1/(h+a)$ and $\theta_1=h+a$. This means that
the vectors $(\theta_0,1)$ and $(1,\theta_1)$ in $\R^2$
are proportional to each other, which is precisely the
resonance condition that would have to be violated to guarantee
infinitely many periodic Reeb orbits.
\end{rem}
\subsection{The self-linking number}
A periodic Reeb orbit in a contact $3$-manifold $(M,\alpha)$
constitutes a transverse knot $K$ for the
contact structure $\xi=\ker\alpha$. When $K$ is homologically trivial in~$M$,
it bounds a Seifert surface~$\Sigma$, over which
the $2$-plane field $\xi$ is trivial. Choose a non-vanishing section $Z$
of $\xi|_{\Sigma}$, and push $K$ in the direction of $Z|_K$
to obtain a parallel copy $K'$ of~$K$.
The self-linking number $\ttsl(K,\Sigma)$ is then defined as
the linking number of $K$ and~$K'$, that is,
the oriented intersection number of $K'$ and~$\Sigma$,
see~\cite[Definition~3.5.28]{geig08}. When the Euler class
of $\xi$ vanishes, the self-linking number is independent of the choice
of Seifert surface, see~\cite[Proposition~3.5.30]{geig08}.
In that case, we write~$\ttsl(K)$ for the self-linking number.

Going back to the contact forms on $S^3$ found via a cut construction
on $S^1\times D^2$, the self-linking number $\ttsl(B)$ of the
binding orbit is defined. As Seifert surface we take the
meridional disc $\Delta_0:=\{0\}\times D^2\subset S^1\times D^2$ as before,
and the trivialisation of $\ker\alpha|_{\{0\}\times D^2}$ given by~$\bfe_1$.
Strictly speaking, we cannot push $B=\{0\}\times\partial D^2$ in the
direction of $\bfe_1$, but one can make sense of this as one
passes to the quotient $S^3=(S^1\times D^2)/\!\!\sim$, and we may as
well perform the homotopical computation in $S^1\times\R^2$.

The parallel knot $B'$ intersects the meridional disc $\Delta_0$
in a single point on the negative $x$-axis, since $\bfe_1$ is
a positive multiple of $\partial_x$ along the $x$-axis.
For $y>0$, $B'$ lies above the $\{s=0\}$-plane; for $y<0$, below.
It follows that the intersection point of $B'$ and $\Delta_0$ is
a negative one, that is, $\ttsl(B)=-1$.

This accords with \cite[Theorem~1.5]{hrsa11}. That theorem
establishes the conditions $\ttsl(P)=-1$ and $\muCZ(P)\geq 3$ as necessary
for a (simply covered) periodic Reeb orbit $P$ to bound a disc-like global
surface of section.  The general assumption there is that $P$ is an
unknotted periodic Reeb orbit in $S^3$ for a contact form
defining the standard contact structure.

\begin{rem}
Much of our discussion carries over to
contact structures on lens spaces, provided we start with
a Hamiltonian function $H_s$ on $S^1\times D^2$ invariant under rotations
of the $D^2$-factor about an angle $2\pi/p$. See~\cite{hls15}
for a dynamical characterisation of universally tight contact structures
on lens spaces.

Also, one may replace $S^1\times D^2$ by $S^1\times\Sigma$, where
$\Sigma$ is any compact surface with boundary. As diffeomorphisms
$\psi\co\Sigma\rightarrow\Sigma$ we may take any Hamiltonian diffeomorphism
whose generating Hamiltonian $H_s$ satisfies criteria as in
Section~\ref{subsection:sufficient}.
\end{rem}
\begin{ack}
We thank Barney Bramham for many useful conversations, especially concerning
the work of Fayad--Katok. We also thank Murat Sa\u{g}lam for
comments on an earlier version of this paper.
This research is part of a project
in the SFB/TRR 191 `Symplectic Structures in Geometry, Algebra and Dynamics',
funded by the DFG.
\end{ack}

\end{document}